\documentclass[10pt]{article}
\usepackage{amsfonts}
\usepackage{amsmath,amsthm}
\usepackage{amssymb}
\usepackage{xcolor,soul}
\usepackage{textcomp}
\usepackage{geometry}
\usepackage{enumitem}
\usepackage[utf8]{inputenc}
\usepackage{todonotes}

\usepackage{graphicx}
\graphicspath{{figures/}}
\usepackage{subcaption}
\usepackage{epstopdf}
\usepackage{float}
\usepackage{tikz}

\definecolor{NewGGreen}{RGB}{0,99,0}

\usetikzlibrary{positioning}
\usetikzlibrary{arrows}
\usetikzlibrary{arrows.meta}
\usetikzlibrary{calc}

\newtheorem{prop}{Proposition}
\newtheorem{lemm}{Lemma}
\newtheorem{remark}{Remark}

\title{A geometric analysis of the SIRS epidemiological model on a homogeneous network}
\author{Hildeberto Jard\'on-Kojakhmetov{$^{1}$}, Christian Kuehn{$^{2}$},\\
	Andrea Pugliese{$^{3}$}, Mattia Sensi{$^{3}$}\\[1em]
	$^1$Faculty of Science and Engineering, University of Groningen\\
	$^2$Department of Mathematics, Technical University of Munich \\
	$^3$Universit\`a degli Studi di Trento \\ }

\begin{document}

\maketitle

\begin{abstract}
We study a fast-slow version of an SIRS epidemiological model on homogeneous graphs, obtained through the application of the moment closure method. We use GSPT to study the model, taking into account that the infection period is much shorter than the average duration of immunity. We show that the dynamics occurs through a sequence of fast and slow flows, that can be described through 2-dimensional maps that, under some assumptions, can be approximated as 1-dimensional maps. Using this method, together with numerical bifurcation tools, we show that the model can give rise to periodic solutions, differently from the corresponding model based on homogeneous mixing.
\end{abstract}

\textbf{Keywords:} fast-slow system, epidemic model, non-standard form, epidemics on networks, bifurcation analysis

\section{Introduction}

Mathematical epidemics modelling is, now more than ever, an important and urgent field to explore. A deep understanding of how diseases evolve and spread can give, and has given, us strategies to contain, treat and even prevent them.\\
Over the years, mathematical modellers have made a variety of different assumptions, in order to obtain a tractable trade-off between simplicity, which allows for more in-depth analysis, and realism, which allows to make more precise predictions.\\
In particular, compartment models build on the core idea that the population can, at any time, be portioned into compartments characterized by a specific state with respect to the ongoing epidemic. The first of such models divides the population into Susceptible, Infected and Recovered individuals, from which the SIR acronym is used. A Susceptible can become Infected ($S \rightarrow I$) by making contact with an already infected individuals, and can then either Recover ($I \rightarrow R$) or die, if we assume the disease to be characterized by permanent immunity after a first infection. If we do not make such an assumption, and allow recovered individuals to become susceptible again ($R \rightarrow S$), we obtain a so called SIRS model. Many more models, with different compartments, have been proposed and analysed in the past, see e.g. \cite{Heth, li1995global, dafilis2012influence, giordano2020sidarthe}.\\
Classical compartmental models are based on the homogeneous mixing assumption, i.e. the assumption that any individual in a population may have contacts with any other. Such an assumption, however, is quite unrealistic for many situations in which the observed population is large, and possibly divided in classes, families or generally sub-populations. 
One possible extension is to subdivide the population into groups, assuming homogeneous mixing within each group, but representing inter-group interactions through a contact matrix \cite{mossong2008social}.
Another possible approach is to take into account the network structure of contacts. Often, epidemic dynamics on a network is analysed only through simulations \cite{lopez2016stochastic, smilkov2014beyond, zhang2013stochastic, castellano2010thresholds, volz2008sir, ganesh2005effect}. The method of pair approximations, introduced in epidemiology by Sat\=o et al. \cite{sato1994pathogen} and Keeling et al. \cite{keeling1997correlation}, allows to build a system of differential equations that retains some aspects of the network structure. The ideas and some applications of the methods are presented in detail in the monograph by Kiss et al. \cite{KMS}.
However, not much analytical progress has been made in the study of the resulting systems, possibly because they are generally rather complex.\\
This paper aims at introducing methods from Geometric Singular Perturbation Theory (GSPT) to analyse these systems, building on the ideas introduced in \cite{jardnkojakhmetov2020geometric}. The difference in time-scales between epidemic spread and demographic turnover, which can be observed in many diseases, is the motivation for the use of techniques from GSPT. We refer to \cite{jardnkojakhmetov2020geometric} for a brief introduction of the techniques we use, or to the references therein, and in particular to \cite{jones1995geometric} and \cite{Kuehn}, for a more detailed explanation.
In particular, we will exploit the \emph{entry-exit function} \cite{de2008smoothness, de2016entry} to analyse the behaviour of the system on its critical manifold, which is characterized by a change in stability over a hyperplane.\\
In this work, we assume homogeneity of the network, in order to obtain analytical results, before validating them numerically. Even with such an assumption, the additional complexity brought by the network structure must be treated properly. In fact, in order to completely describe the evolution of a network in time, one needs to have an equation for each possible state of its nodes, one for each possible state of its edges (along which the epidemic spreads), one for each possible state of triples, i.e. three nodes connected by two edges, and so on. This procedure, however, would generate an infinite system of ODEs, which would once again be hardly treatable with analytical tools. In order to overcome this difficulty, one can apply the so-called \emph{moment closure} \cite{kuehn2016moment, KMS}, i.e. approximation formulas which allow us to truncate the dimension of the objects we want to analyse. If we truncate at the node level, we lose the network structure, and we recover a homogeneously mixing system. Instead, we truncate at the edge level, using the pair approximation discussed above, and analyse the system which derives from this choice.\\
To our knowledge, there are relatively few articles in which GSPT has been applied rigorously to epidemics models \cite{rocha2016understanding, jardnkojakhmetov2020geometric, heesterbeek1993saturating, zhang2009singular, brauer2019singular, wang2014dynamical}; however, for most infectious diseases, the presence of different time scales is natural. Moreover, though a SIR model on networks has been studied with moment closure already \cite{bidari2016solvability, KMS}, the SIRS extension has not. Likewise, a thorough bifurcation analysis on compartment models such as the one we analyse in this paper is not present in the literature.\\
The additional feature of the network structure, even in its most simplified version, i.e. homogeneous network, unravels new dynamics for the SIRS system we study. Indeed, there exists a set in the parameter space which allows the system to exhibit a stable limit cycle. To complement the bifurcation analysis, we extend the geometrical argument from \cite{jardnkojakhmetov2020geometric} to the higher dimensional system we study, providing additional justification for the existence of stable limit cycles.\\
It is worth noticing that the model we study is not globally in fast-slow standard form; as in \cite{jardnkojakhmetov2020geometric, kuehn2015multiscale, kosiuk2016geometric}, the fast-slow dynamics are only evident in specific regions of the phase space, in which a local change of coordinates brings the system to a standard two time scales form. In particular, we refer to the very recent monograph \cite{wechselberger2020geometric}, in which the properties of perturbed systems in non-standard form  are thoroughly analysed.\\
The paper is structured as follows: in Section \ref{sec:model}, we recall the derivation of the model, and introduce the moment closure technique. In Section \ref{sec:analysis}, we obtain analytical results on the model, in particular on the fast and slow limit systems and on the application of the entry-exit function. In Section \ref{sec:bifurc}, we perform a bifurcation analysis and numerical exploration of the model. Finally, in Section \ref{sec:conclusions}, we conclude with a summary of the results, and with possible research outlooks.

\section{Formulation of the SIRS model on a network}\label{sec:model}

In this section we describe and propose an SIRS model for epidemics on graphs, building on the model proposed in \cite[Sec. 4.2.2]{KMS}. We are interested in the graph generalization of the model studied in \cite{jardnkojakhmetov2020geometric}, in order to drop the homogeneous-mixing hypothesis, under which we assumed that each individual in the population could have contacts with any other. We then assume loss of immunity to be slower, compared to the other rates (this is the case e.g. for pertussis \cite{dafilis2012influence, lavine2011natural}, and it could potentially be true for the recent SARS-CoV-2 \cite{kissler2020projecting, randolph2020herd}); this assumption brings the model to a non-standard perturbed system of ODEs, which we study with techniques from GSPT. 

\subsection{The model}

The construction of the model is essentially what is presented in detail in \cite[Ch. 4]{KMS}, extended to the SIRS case. For ease of reading, we briefly repeat the whole method.\\
We consider a network of $N$ nodes, with $N$ large, representing the individuals of a population, and we assume this network to be homogeneous, meaning that each node has fixed degree $n \in \mathbb{N}_{\geq 2}$, representing the number of direct neighbours each individual has. We assume the network to be undirected and connected, meaning that, given any two nodes in the network, there is a finite sequence of edges (or an \textit{undirected path}) which starts in the first and ends in the second.\\
Each node can be in three states, namely $S$ (susceptible), $I$ (infected) or $R$ (recovered). We will indicate the number of each state at time $t$ with $[\cdot](t)$; we stress the distinction between the notation $X$, indicating a state, and $[X]$, indicating the number of individuals in the state $X$. We indicate the number of edges connecting a node in state $X$ to one in state $Y$ at time $t$ with $[XY](t)$ for all $t\geq 0$. We distinguish between an edge $XY$, counted starting from a node in state $X$, and the same edge counted starting from the other node in state $Y$, for a reason of conserved quantities, namely (\ref{eqn:const1}), (\ref{eqn:const2}) and (\ref{eqn:const3}) to be defined below. For example, we count the number of edges $SI$ by ``visiting'' each node in state $S$, and counting all its neighbours in state $I$, then summing over all the nodes in state $S$; this implies that, at all times, by definition, $[SI]=[IS]$. The edges connecting a node with another in the same state, such as $SS$, hence, will always be counted twice.\\
Infection can only spread if a node in state $S$ is connected to a node in state $I$ through an edge $SI$; we denote the infection rate with $\beta\geq 0$. Nodes in state $I$ recover, independently from their neighbours, at a rate $\gamma>0$; and nodes in state $R$ lose their immunity, again independently from their neighbours, at a slow rate $\epsilon$, with $0<\epsilon\ll \beta,\gamma$. Based upon these modelling assumptions, it is then straightforward to prove using the master equation of the epidemic model, that one obtains the following system of ODEs:
\begin{align}
\begin{split}
[S]'={}&-\beta [SI]+ \epsilon [R],\\
[I]'={}&\beta [SI]- \gamma [I],\\
[R]'={}&\gamma [I]- \epsilon [R].\\
\end{split}
\label{eqn:nodes1}
\end{align}
From our assumptions, the sum of $[S]+[I]+[R]\equiv N$ is conserved at all times; we normalize by dividing both nodes and edges by $N$, and we do not rename the new variables, which now indicate the density of nodes, and a rescaled fraction of edges, in each state. Now $[S]+[I]+[R]\equiv 1$, so we can reduce the dimension of system (\ref{eqn:nodes1}) by removing $[R]$, obtaining the system
\begin{align}
\begin{split}
[S]'={}&-\beta [SI]+ \epsilon  (1-[S]-[I]),\\
[I]'={}&\beta [SI]- \gamma [I].\\
\end{split}
\label{eqn:nodes2}
\end{align}
In order to fully describe the dynamics of the system, we need an ODE for $[SI]$ as well.
To understand how the number of edges $[SI]$ evolve in time, we need to consider the role of triples, as exemplified in Figure \ref{fig:triple}. A triple is a path of length $2$ through a central node in state $Y$, connected to two nodes in state $X$ and $Z$, respectively; we indicate such a triple with $XYZ$. The positions of $X$ and $Z$ are interchangeable, and the most important node is the central one, as we will explain shortly.
\begin{figure}[H]\centering
	\begin{tikzpicture}
	\node[draw,circle,thick,minimum size=.75cm] (s1) at (0,-1) {$S$};
	\node[draw,circle,thick,minimum size=.75cm] (i1) at (1,0) {$S$};
	\node[draw,circle,thick,minimum size=.75cm] (r1) at (2,-1) {$I$};
	\node at (0.3,-0.3) {$SS$};
	\node at (1.7,-0.3) {$SI$};
	\draw[thick] (s1)--(i1);
	\draw[thick] (i1)--(r1);
	\draw [-{Latex[length=2.mm, width=1.5mm]},thick] (2.5,-0.25)--(3.5,-0.25)  node[above, midway]{$\beta$};	
	\node[draw,circle,thick,minimum size=.75cm] (s2) at (4,-1) {$S$};
	\node[draw,circle,thick,minimum size=.75cm] (i2) at (5,0) {$I$};
	\node[draw,circle,thick,minimum size=.75cm] (r2) at (6,-1) {$I$};
	\node at (4.3,-0.3) {$SI$};
	\node at (5.7,-0.3) {$II$};
	\draw[thick] (s2)--(i2);
	\draw[thick] (i2)--(r2);
	\end{tikzpicture}
	\caption{Example of the role of triples. The rightmost edge (of the triple on the left) turns from $SI$ to $II$ because the infection spreads to the central node; the leftmost edge turns from $SS$ to $SI$ because it belongs to a triple $SSI$.}
	\label{fig:triple}
\end{figure}
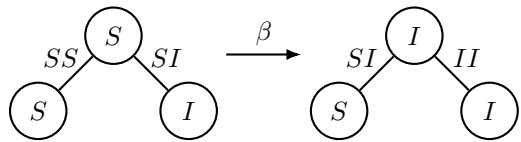	
The only change of the system which depends on the presence of a specific edge is the contagion which brings $SI \rightarrow II$. Direct neighbours of a node in the state $S$ which get infected, i.e. the node $X$ in a triple $XSI$, see their edge $XS$ change to $XI$ due to their belonging to the triple. The two other possible changes in the system, namely the recovery (a node in state $I$ becoming $R$, which happens at a rate $\gamma$) and the loss of immunity (a node in state $R$ becoming $S$, which happens at a rate $\epsilon$) only happen at a node level, so the only nodes which see this change are the direct neighbours of the node changing state, and we do not need to consider their belonging to a triple.\\
\begin{figure}[H]\centering
	\begin{tikzpicture}
	\node at (0,0){
		\includegraphics[scale=0.35]{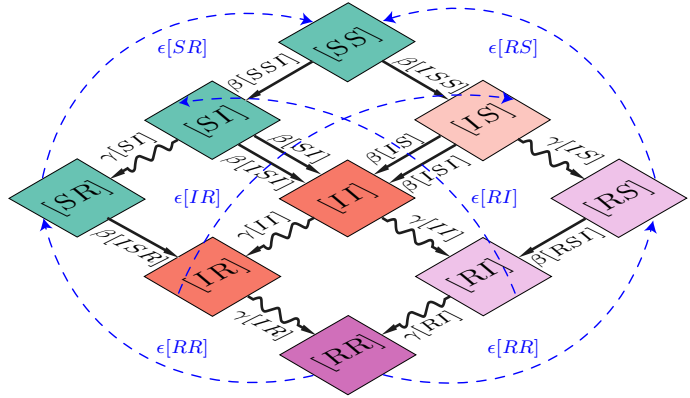}
	};
	\begin{scope}[every node/.append style={
	    yslant=0.5,xslant=-1},yslant=0.5,xslant=-1
	  ]
	\node at (2.05,2.05) {$[SS]$};	  
	\node at (0,0) {$[II]$};
	\node at (0.15,1.9) {$[SI]$};
	\node at (1.95,0.15) {$[IS]$};
	\node at (-1.8,1.8) {$[SR]$};	
	\node at (1.8,-1.8) {$[RS]$};	
	\node at (-0.15,-1.9) {$[RI]$};
	\node at (-1.9,-0.15) {$[IR]$};
	\node at (-2.05,-2.05) {$[RR]$};	
	\node at (1.1,2.25) {\scriptsize $\beta[SSI]$};
	\node at (-0.8,2.15) {\scriptsize $\gamma[SI]$};	
	\node at (-0.95,0.2) {\scriptsize $\gamma[II]$};	
	\node at (0.95,0.35) {\scriptsize $\beta[IS]$};	
	\node at (0.9,-0.25) {\scriptsize $\beta[ISI]$};	
	\node at (0.8,-2.1) {\scriptsize $\beta[RSI]$};
	\node at (-1.15,-2.25) {\scriptsize $\gamma[RI]$};
	\end{scope}
	\node at (-2,0) {\scriptsize \color{blue} $\epsilon [IR]$};
	\node at (1.95,0) {\scriptsize \color{blue} $\epsilon [RI]$};
	\node at (-2.2,-2) {\scriptsize \color{blue} $\epsilon [RR]$};
	\node at (2.2,-2) {\scriptsize \color{blue} $\epsilon [RR]$};
	\node at (2.2,2) {\scriptsize \color{blue} $\epsilon [RS]$};
	\node at (-2.2,2) {\scriptsize \color{blue} $\epsilon [SR]$};
	\begin{scope}[every node/.append style={
	    yslant=-0.5,xslant=1},yslant=-0.5,xslant=1
	  ]
	\node at (-1.1,2.2) {\scriptsize $\beta[ISS]$};	  
	\node at (0.85,2.15) {\scriptsize $\gamma[IS]$};	
	\node at (-1,0.35) {\scriptsize $\beta[SI]$};
	\node at (-0.95,-0.25) {\scriptsize $\beta[ISI]$};	
	\node at (-0.8,-2.1) {\scriptsize $\beta[ISR]$};	
	\node at (1.1,-2.25) {\scriptsize $\gamma[IR]$};	
	\node at (0.95,0.2) {\scriptsize $\gamma[II]$};
	\end{scope}
\end{tikzpicture}
	\caption{Complete description of the edges dynamics considering edges and triples. Straight lines: infections; wobbly lines: recovery; dashed lines: loss of immunity. The base diagram is the same which appears in \cite{KMS}, to visually describe their SIR model; the new, slow dynamics in our model are the dashed blue arrows, symbolizing loss of immunity.}	
	\label{fig:flux}
\end{figure}
For clarity, we fix a lexicographic order $S \prec I \prec R$ for nodes and edges, and write the explicit equations for the edges which follow this order only. If we take into account all the triples with a central node in state $S$ and at least one node $I$, which could infect the central one (as described in Figure \ref{fig:flux}), we obtain the following system of ODEs, which describes the evolution in time of nodes and edges:
\begin{align}
\begin{split}
[S]'={}&-\beta [SI]+ \epsilon  (1-[S]-[I]),\\
[I]'={}&\beta [SI]- \gamma [I],\\
[SS]'={}&2\epsilon [SR]-2\beta[SSI],\\
[SI]'={}&-(\gamma+\beta)[SI]+\epsilon  [IR]+\beta[SSI]-\beta[ISI],\\
[SR]'={}&\gamma [SI] -\epsilon  [SR] +\epsilon [RR]- \beta [ISR],\\
[II]'={}&2\beta[SI]-2\gamma [II]+2\beta[ISI],\\
[IR]'={}&\gamma [II]-(\gamma +\epsilon ) [IR]+\beta[ISR],\\
[RR]'={}& 2\gamma [IR]-2\epsilon  [RR].
\end{split}
\label{eqn:sist1}
\end{align}
Notice the 2 which multiplies the right hand sides of edges connecting nodes in the same state: as we mentioned above, they are always counted twice, whether they are created or lost. To fully describe the system, we would then need to have ODEs for triples, quadruples, etc. Instead, we proceed as in \cite{KMS}, and apply \textit{moment closures}.

\subsection{Moment closures}
Moment closure methods are approximation methods used in many contexts, in order to reduce large (or infinite) dimensional systems of equations to a smaller finite dimension \cite{kuehn2016moment}. Proceeding as in \cite[Sec. 4.2]{KMS}, one can approximate the edges as functions of the nodes, or triples as functions of nodes and edges. If we choose the first option, assuming independence between the state of nodes, we can approximate all edges as follows:
\begin{equation}
[XY] \approx n [X][Y].
\label{eqn:clos1}
\end{equation}
This implies that we lose the network structure and, up to rescaling the infection parameter by $\tilde{\beta}=n\beta$, we recover the SIRS system already studied in \cite{jardnkojakhmetov2020geometric}.
\begin{lemm}\label{lemm:l} Consider \eqref{eqn:nodes2}. Applying approximation \eqref{eqn:clos1} and rescaling $\tilde{\beta}=n\beta$, one recovers the \emph{SIRS} system studied in \cite{jardnkojakhmetov2020geometric}, which is characterized by an asymptotic stability of the endemic equilibrium for orbits starting in the set $\left\{ (S,I,R)\in\mathbb R^3_{\geq0}\, | \, \, S+I+R\leq 1, I>0 \right\}$.
\end{lemm}
Instead, in this work we choose to apply the second order approximation, and hence we approximate each triple with the formula given in equation (4.6) of \cite{KMS}, namely
\begin{equation}
[XYZ]\approx \frac{n-1}{n} \frac{[XY][YZ]}{[Y]}.
\label{eqn:clos2}
\end{equation}
This approximation is based on the conditional independence between the states of neighbors of a node, using a counting argument, which for clarity we recall from \cite{KMS}. The total number of edges starting from a node in state $Y$ is $n[Y]$, while the total number of edges in state $XY$ is $[XY]$; this means that a fraction $[XY]/(n[Y])$ of edges starting from a node in state $Y$ reach a node in state $X$. With the same procedure, we obtain a fraction $[YZ]/(n[Y])$ of edges which connect a node in state $Y$, from which we start, with one in state $Z$. Hence, selecting a node in state $Y$ and two of his direct neighbours $u$ and $v$, and using the conditional independence of $u$ and $v$, the probability of them forming a triple $XYZ$ is $[XY][YZ]/(n^2[Y]^2)$. Combinatorics tell us there are $n(n-1)$ ways of picking $u$ and $v$, and $[Y]$ nodes in state $Y$; multiplying $n(n-1)\cdot[Y]\cdot[XY][YZ]/(n^2[Y]^2)$, we obtain formula (\ref{eqn:clos2}).

\section{Analysis of the model}\label{sec:analysis}

In this section we present the pair approximation SIRS model, and give our main analytical results. First, we are going to reduce the dimension of the system, exploiting three conserved quantities. Second, we are going to introduce a formulation for the basic reproduction number for the system, and we describe the behaviour of the fast limit system. Third, we are going to derive the equilibria of the system in the biologically relevant region, and we show that the slow manifold of our perturbed system is exponentially close to the critical manifold. Last, we are going to rescale the system in an $\mathcal{O}(\epsilon)$-neighbourhood of the critical manifold, with a scaling similar to the one proposed in \cite{jardnkojakhmetov2020geometric}, and we apply the entry-exit procedure.\\
Throughout the analysis, we notice that the parabola $[SS]=n[S]^2$, i.e. approximation (\ref{eqn:clos1}) applied to the edges in state $[SS]$, on the critical manifold is of particular importance for the dynamics.

\subsection{Fast-slow model}

In this section, we derive the system we will study for the remainder of the article, applying moment closure to  \eqref{eqn:sist1} and reducing its dimension.

Applying approximation \eqref{eqn:clos2} to every triple in system \eqref{eqn:sist1}, we obtain the following singularly perturbed autonomous system in non-standard form:
\begin{align}
\begin{split}
[S]'={}&-\beta [SI]+ \epsilon   (1-[S]-[I]),\\
[I]'={}&\beta [SI]- \gamma [I],\\
[SS]'={}&2\epsilon  [SR]-2\beta\frac{n-1}{n}\frac{[SS][SI]}{[S]},\\
[SI]'={}&-(\gamma+\beta)[SI]+\epsilon   [IR]+\beta\frac{n-1}{n}[SI]\bigg(\frac{[SS]}{[S]}-\frac{[SI]}{[S]}\bigg),\\
[SR]'={}&\gamma [SI] -\epsilon   [SR] +\epsilon  [RR]- \beta \frac{n-1}{n}\frac{[SI][SR]}{[S]},\\
[II]'={}&2\beta[SI]-2\gamma [II]+2\beta\frac{n-1}{n} \frac{[SI]^2}{[S]},\\
[IR]'={}&\gamma [II]-(\gamma +\epsilon  ) [IR]+\beta \frac{n-1}{n}\frac{[SI][SR]}{[S]},\\
[RR]'={}& 2\gamma [IR]-2\epsilon  [RR],
\end{split}
\label{eqn:nonstand}
\end{align}
in which, as from our assumptions, the processes of infection and recovery are fast, and the process of loss of immunity is slow. By construction, the sum of all the edges starting from a node in the state $[S]$ (or $[I]$ or $[R]$, respectively) is equal to
\begin{subequations}
\begin{align}
[SS]+[SI]+[SR]&=n[S],\label{eqn:const1}\\
[SI]+[II]+[IR]&=n[I],\label{eqn:const2}\\
[SR]+[IR]+[RR]&=n[R],\label{eqn:const3}
\end{align}\label{eqn:const}%
\end{subequations}
which allows us to remove the equation governing $[SR]$ (and $[IR]$ and $[RR]$, respectively). This can be checked by carefully computing the difference of the derivatives of the right hand side(s) and the left hand side(s) of (\ref{eqn:const}). By doing so, we reduce the dimension of the system, obtaining
\begin{subequations}
\begin{align}
[S]'={}&-\beta [SI]+ \epsilon  (1-[S]-[I]),\label{eqn:sistS}\\
[I]'={}&\beta [SI]- \gamma [I],\label{eqn:sistI}\\
[SS]'={}&2\epsilon (n[S]-[SS]-[SI])-2\beta\frac{n-1}{n}\frac{[SS][SI]}{[S]},\label{eqn:sistSS}\\
[SI]'={}&-(\gamma+\beta)[SI]+\epsilon  (n[I]-[SI]-[II])+\beta\frac{n-1}{n}[SI]\bigg(\frac{[SS]}{[S]}-\frac{[SI]}{[S]}\bigg),\label{eqn:sistSI}\\
[II]'={}&2\beta[SI]-2\gamma [II]+2\beta\frac{n-1}{n} \frac{[SI]^2}{[S]}.\label{eqn:sistII}
\end{align}\label{eqn:sist2}%
\end{subequations}
The \textit{basic reproduction number} $R_0$ can be obtained \cite[p. 140]{KMS} for the limit as $\epsilon \rightarrow 0$ of system (\ref{eqn:sist2}) as
\begin{equation}
R_0 =\frac{\beta(n-2)}{\gamma}.
\label{eqn:basicrep}
\end{equation}
We notice that, for (\ref{eqn:basicrep}) to be well-defined and dependent on the parameters of the system, we need $n>2$. The equality $n=2$ describes the very special case of a ring network, i.e., a connected network in which all nodes have exactly two neighbours. In the remainder of the paper we assume $R_0>1$ and $n>2$.
\begin{remark}\label{othererr}
	We notice that the threshold $R_0 \lessgtr 1$ in (\ref{eqn:basicrep}) is equivalent to
\begin{equation}
	R_1:=\frac{\beta(n-1)}{\beta+\gamma}\lessgtr 1 \iff 	R_2:=\frac{\beta n}{2\beta+\gamma}\lessgtr 1,
\label{eqn:otherbas}
\end{equation}
	since they all correspond to $\beta(n-2)\lessgtr \gamma$. A formula corresponding to $R_1$ is given in \cite{KMS}, shortly after the definition of $R_0$.
\end{remark}
 We notice that $R_1$ has a much more intuitive biological interpretation than $R_0$. Consider a network with all the nodes in susceptible state $S$, except one in state $I$. Consider one of the $n$ edges in state $IS$: this could either transition to $RS$, at a rate $\gamma$, and the epidemics would die out immediately, or spread the infection to the node in state $S$, at a rate $\beta$, and become an edge $II$. If the latter happens, with probability $\beta/(\beta+\gamma)$, $(n-1)$ new edges move to state $SI$; hence, $R_1$ can be interpreted in the classical meaning of ``the number of edges infections caused by one infected edge in an otherwise susceptible population''. Recall that the disease spreads only through edges $SI$ (or $IS$, equivalently), so their number should be the quantity we measure in order to quantify the contagiousness of the disease; an edge $II$ can not be used to spread the disease.\\
Now we compute the basic reproduction number $R_1$ for system (\ref{eqn:sist2}) and $\epsilon>0$ sufficiently small.

\begin{prop}\label{arrnod} The basic reproduction number $R_1$ for system (\ref{eqn:sist2}) is given by
\begin{equation}\label{eqn:truebasicrep}
R_1=\frac{\beta(n-1)(\gamma+\epsilon)}{\gamma(\gamma+\beta+\epsilon)}.
\end{equation}
\end{prop}

\begin{proof}
We use the method first introduced in \cite{diekmann1990definition}, and then generalized in \cite{van2002reproduction} (see also \cite{diekmann2010construction}). We linearize system (\ref{eqn:nonstand}) at the disease free equilibrium 
	$$
	([S],[I],[SS],[SI],[SR],[II],[IR],[RR])=(1,0,n,0,0,0,0,0)
	$$
focusing on the infected compartments. In this case we choose as  variables describing the infected compartments $[SI]$, $[II]/2$ and $[IR]$ obtaining
	$$
	\begin{pmatrix}
	[SI]\\
	[II]/2\\
	[IR]
	\end{pmatrix}'
	=A	
	\begin{pmatrix}
	[SI]\\
	[II]/2\\
	[IR]
	\end{pmatrix},
	$$
	with the matrix $A$ given by
	$$
	A=	
	\begin{pmatrix}
	\beta(n-2)-\gamma & 0 & \epsilon\\
	\beta & -2\gamma & 0\\
	0 & 2\gamma & -(\gamma+\epsilon)
	\end{pmatrix}.
	$$
We split $A=M-V$, with $V$ invertible, $M$ and $V^{-1}$ having non-negative entries. There are clearly many ways of doing that, but the preferred splitting is such that  $M$ and $V$ can be interpreted as the \emph{transmission} (i.e. relative to new infections) and \emph{transition} matrix (i.e. relative to any other change of state), respectively. Then, we compute
	$$
	R_1 = \rho(M V^{-1}),
	$$
	where $\rho$ indicates the spectral radius of a matrix. The choice for the two matrices is
		$$
	M=	
	\begin{pmatrix}
	\beta (n-1) & 0 & 0\\
	0 & 0 & 0\\
	0 & 0 & 0
	\end{pmatrix}, \quad
	V=
		\begin{pmatrix}
	\gamma+\beta & 0 &-\epsilon\\
	-\beta & 2\gamma & 0\\
	0 & -2\gamma & \gamma+\epsilon
	\end{pmatrix}.
	$$
	It can easily be checked, then, that $V^{-1}$ has non-negative entries, and that, since $M V^{-1}$ has two rows of zeros,
	\begin{equation}
\label{R1_eps}
\rho(M V^{-1})=(M V^{-1})_{1,1}=
	R_1:=\frac{\beta(n-1)(\gamma+\epsilon)}{\gamma(\gamma+\beta+\epsilon)}.
\end{equation}
This finishes the proof.
\end{proof}
\begin{remark}
	The perturbed $R_1$ given in (\ref{eqn:truebasicrep}) has a similar biological interpretation for the perturbed system to the one given for the corresponding $R_1$ (\ref{eqn:otherbas}) of the limit system as $\epsilon \rightarrow 0$. \\
We need to compute $R_1$, the average number of $SI$ edges produced by an $SI$ edge in a totally susceptible population; as in the previous case, an edge $SI$ will become an edge $II$ with probability $\beta/(\beta+\gamma)$, producing in this case $n-1$ edges $SI$; however, the original edge $II$, after having become $IR$ can become again an $IS$ edge with probability $\epsilon/(\epsilon+\gamma)$. After having returned $SI$, the edge will produce other $R_1$ $SI$ edges, since the pairwise model does not consider higher order correlation and does not ``remember'' that the neighbours of $S$ had already been infected once. Hence
	$$ R_1 = \frac{\beta}{ \beta+\gamma}\left( n-1 + \frac{\epsilon}{\epsilon + \gamma } R_1\right), $$
	from which one obtains \eqref{R1_eps}.\\
	Through this argument, we see that threshold for the SIRS model is different from the one for the SIR model, while in the homogeneous mixing case the two coincide.
\end{remark}
\begin{lemm}\label{convex}
System (\ref{eqn:sist2}) is well posed in the convex set 
\begin{equation}
\begin{split}
\Delta = \{& ([S],[I],[SS],[SI],[II])\in \mathbb{R}^5_{\geq 0} | \\
& 0 \leq [S]+[I] \leq 1 \} \cap \{ 0 \leq [SS]+[SI] \leq n[S], 0 \leq [SI]+[II]\leq n[I] \}.
\end{split}
\label{eqn:simp}
\end{equation}
The set is forward invariant under the flow of (\ref{eqn:sist2}), for $\epsilon \geq 0$, so that solutions of (\ref{eqn:sist2}) are global in time.
\end{lemm}
\begin{proof}
Apparently the right-hand side of (\ref{eqn:sist2}) has a singularity at $[S] = 0$; however, in the set $\Delta$, the terms $[SI]/[S]$ and $[SS]/[S]$ are both bounded by $n$, so that the right-hand side is indeed Lipschitz. Hence, system (\ref{eqn:sist2}) has a local solution.
Furthermore, it can be easily checked that the system is forward invariant by showing that the flow is pointing inwards on the boundary of $\Delta$. Hence, solutions of system (\ref{eqn:sist2}) are global in time.
\end{proof}

\subsection{Fast limit}\label{subsec:fast}
In this section, we study the fast subsystem (or layer equations) corresponding to the limit of system \eqref{eqn:sist2} as $\epsilon\rightarrow 0$ on the fast time scale. Hence, we have to take the limit $\epsilon \rightarrow 0$ in system \eqref{eqn:sist2}, to obtain the layer equations
\begin{subequations}
\begin{align}
[S]'={}&-\beta [SI],\label{eqn:layer1}\\
[I]'={}&\beta [SI]- \gamma [I],\label{eqn:layer2}\\
[SS]'={}&-2\beta\frac{n-1}{n}\frac{[SS][SI]}{[S]},\label{eqn:layer3}\\
[SI]'={}&-(\gamma+\beta)[SI]+\beta\frac{n-1}{n}[SI]\bigg(\frac{[SS]}{[S]}-\frac{[SI]}{[S]}\bigg),\label{eqn:layer4}\\
[II]'={}&2\beta[SI]-2\gamma [II]+2\beta\frac{n-1}{n} \frac{[SI]^2}{[S]}.\label{eqn:layer5}
\end{align}\label{eqn:layer}%
\end{subequations}
For ease of notation, we introduce
\begin{equation*}
\begin{split}
[\cdot]_0={}&[\cdot](0),\\
[\cdot]_{\infty}={}&\lim_{t \rightarrow +\infty}[\cdot](t).\\
\end{split}
\end{equation*}
In the fast dynamics, the susceptible population can only decrease, and eventually the infected population will not have any more susceptibles to ``recruit'' and will decrease as well. In particular, we prove the following:
\begin{prop}\label{prop:fast} Consider system (\ref{eqn:layer});
$[S]$ and $[SS]$ are decreasing for all $t\geq 0$, and they tend to positive constants $[S]_{\infty}$ and $[SS]_{\infty}$. The variables $[I]$, $[SI]$, $[II]$ and $[IR]$ all have the limit $[I]_{\infty}=[SI]_{\infty}=[II]_{\infty}=[IR]_{\infty}=0$.
\end{prop}

\begin{proof}
We proceed to show the claims of the proposition: for $[SS]$ (and implicitly for $[SR]$, referring to (\ref{eqn:const1})), we give the limit value as a function of $[S]_\infty$, $[S]_0$ and $[SS]_0$. We introduce the auxiliary variables $u:=\frac{[SI]}{[S]}$ and $v:=\frac{[SS]}{[S]}$. From (\ref{eqn:layer1}) and (\ref{eqn:layer4}) we see that
\begin{equation*}
u'=-(\gamma+\beta)u+\beta u\bigg(\frac{n-1}{n}v+\frac{1}{n}u\bigg),
\end{equation*}
while from (\ref{eqn:layer1}) and (\ref{eqn:layer3}) we see that
\begin{equation}
v'=-\beta \frac{n-2}{n}uv.
\label{eqn:vprime}
\end{equation}
From our analysis, for any initial point we have $0\leq [SS]+[SI]\leq n [S]$ and $[SS],[SI]\geq 0$. This implies that, for all times
\begin{equation*}
u\geq 0, \quad v\geq 0, \quad u+v\leq n.
\end{equation*}
Note that from (\ref{eqn:vprime}) $v$ is clearly decreasing for $n>2$, and we see that
\begin{equation}
u'+v'=-\gamma u -\beta u \bigg(1-\frac{u+v}{n}\bigg)<0.
\label{eqn:uplusvprime}
\end{equation}
Recall Lemma \ref{convex}, which implies $v\geq0$; if $v=0$, then $[SS]=0$, and from equation (\ref{eqn:layer3}) we observe that $[SS]$ will not change, so $0$ is its corresponding limit value. Assume then $v>0$: since $v'<0$, $v \rightarrow v_\infty$ monotonically as $t \rightarrow \infty$, and since $0\leq u+v\leq n$, this implies that $u \rightarrow u_\infty$ as $t \rightarrow \infty$ as well. Then we notice that
\begin{equation}
0<-\int_{0}^{\infty}(u'(z)+v'(z))~\text{d}z=u_0+v_0- u_\infty - v_\infty<\infty.
\label{eqn:uplusvint}
\end{equation}
We notice that we can rewrite (\ref{eqn:uplusvint}) using (\ref{eqn:uplusvprime}) and obtain
\begin{eqnarray}
\infty&>&-\int_{0}^{+\infty}(u'(z)+v'(z))\text{d}z=\int_{0}^{+\infty}\bigg(\gamma u(z) +\beta u(z) \bigg(1-\frac{1}{n}(u(z)+v(z))\bigg)\bigg)~\text{d}z\\
&>&\gamma \int_{0}^{+\infty}u(z)~\text{d}z.
\label{eqn:uminint}
\end{eqnarray}
This means that
\begin{equation}
\int_{0}^{+\infty}u(z) \text{d}z<+\infty \implies u_\infty=0,
\label{eqn:uint}
\end{equation}
which implies, recalling that $[SI]=u[S]$ and $[S]_\infty <\infty$, that $[SI]_\infty =0$. 
We can now rewrite (\ref{eqn:layer1}) as
\begin{equation*}
[S]'=-\beta u [S],
\end{equation*}
which implies 
\begin{equation}
[S]_\infty = [S]_0 \exp\bigg(-\beta \int_{0}^{+\infty}u(z) \text{d}z\bigg)>0,
\label{eqn:Sinfty}
\end{equation}
Similarly, using (\ref{eqn:vprime}), we can show that 
\begin{equation}
v_\infty = v_0 \exp\bigg(-\beta\frac{n-2}{n} \int_{0}^{+\infty}u(z) \text{d}z\bigg)>0,
\label{eqn:vinfty}
\end{equation}
which implies, using (\ref{eqn:uint}) and (\ref{eqn:Sinfty}), and recalling that $[SS]=v[S]$, that $[SS]_\infty >0$.
In particular, combining (\ref{eqn:Sinfty}) and (\ref{eqn:vinfty}), we can write
\begin{equation}
[SS]_\infty = [SS]_0 \bigg( \frac{[S]_\infty}{[S]_0}\bigg)^{\frac{2n-2}{n}}.
\label{eqn:SSinfty}
\end{equation}
We notice, from (\ref{eqn:const1}), that this implies that $[SR]$ converges to a non-negative limit as well.
Combining (\ref{eqn:layer1}) and (\ref{eqn:layer2}) as above, we show that $[I]$ vanishes as $t\rightarrow \infty$ as well:
\begin{equation*}
[S]'+[I]'=-\gamma[I]<0.
\end{equation*}
Since $[S]\rightarrow[S]_\infty$ as $t\rightarrow +\infty$, also $[I]\rightarrow [I]_\infty$. Proceeding as in (\ref{eqn:uminint}), it can be shown that $[I]_\infty=0$. This yields, by (\ref{eqn:const2}), that $[II]_\infty=0$ and $[IR]_\infty=0$.
\end{proof}
\noindent 
\begin{remark} A relation between $[SS](t)$ and $[S](t)$ for system (\ref{eqn:layer}) analogous to (\ref{eqn:SSinfty}) holds for all $t$. Indeed, noticing that
$$
V(t)=\ln([SS](t))-2\frac{n-1}{n} \ln([S](t)),
$$
is a constant of motion for system (\ref{eqn:layer}), we observe that for any $t\geq 0$ the relation
\begin{equation}\label{eqn:SSofSt}
[SS](t)=[SS]_0 \bigg(\frac{[S](t)}{[S]_0}\bigg)^{\frac{2n-2}{n}},
\end{equation}
holds.
\end{remark}
The equilibria of the limit system are all of the form $[S]=S^* \in [0,1]$, $[I]=0$, $[R]=1-S^*$; $[SS]=SS^* \geq 0$, $[SI]=0$, $[SR]=SR^*\geq0$, $[II]=0$, $[IR]=0$, $[RR]=RR^*\geq0$ with $SS^* + SR^* =nS^*$ and $SR^* + RR^* =n(1-S^*)$; i.e., they lie on the critical manifold (\ref{eqn:critical}).\\
The eigenvalues of the linearization of system (\ref{eqn:sist2}) on the critical manifold  
\begin{equation}
\mathcal{C}_0:=\{ ([S],[I],[SS],[SI],[II])\in \mathbb{R}^5_{\geq 0} | [I]=[SI]=[II]=0\},
\label{eqn:critical}
\end{equation}
are
\begin{equation*}
\lambda_1=\lambda_2=0,
\end{equation*}
corresponding to the slow variables $[S]$ and $[SS]$,
\begin{equation*}
\lambda_3=\frac{\lambda_4}{2}=-\gamma<0,
\end{equation*}
and
\begin{equation}
\lambda_5=\beta\frac{(n-1)[SS]}{n[S]}-(\gamma+\beta)
\label{eqn:eigwsign}.
\end{equation}
In particular, $\lambda_5$ changes sign on the hyperplane $\beta(n-1)[SS]-n(\gamma+\beta)[S]=0$. We notice that $\beta (n-1)>0$, since we suppose $n>2$.\\
Considering (\ref{eqn:eigwsign}), we define the \emph{loss of hyperbolicity line} on the critical manifold $\mathcal{C}_0$
\begin{equation}
[SS]=\frac{n(\beta+\gamma)}{\beta(n-1)}[S]=:L[S].
\label{losshyp}
\end{equation}
We now give a closed formula for the value of $[S]_\infty$.
\begin{prop}\label{prop:entryy}
	Consider a generic initial condition $([S]_0,[SS]_0)$ in the repelling region of $\mathcal{C}_0$, i.e. satisfying $R_1[S]_0>1$ and $[SS]_0>L[S]_0$. The entry point $[S]_\infty$ is given as the unique zero smaller than $[S]_0$ of the function
	\begin{equation}
	H(x)=n\frac{\beta+\gamma}{\beta}(x^{\frac{1}{n}}-[S]_0^{\frac{1}{n}})-[SS]_0([S]_0^{\frac{2}{n}-2}x^{1-\frac{1}{n}}-[S]_0^{\frac{1}{n}-1}).
	\label{eqn:Hsinf}
	\end{equation}
\end{prop}
\begin{proof}
	We proceed as in \cite[Sec. 3]{bidari2016solvability}. From our assumptions, $[SI](0)=\mathcal{O}(\epsilon)$.\\ 
	Combining (\ref{eqn:layer1}), (\ref{eqn:layer4}) and (\ref{eqn:SSofSt}), we obtain
	\begin{equation*}
	[SI]'-\frac{n-1}{n}\frac{[SI]}{[S]}[S]'=\frac{\beta+\gamma}{\beta}[S]'-\frac{n-1}{n}[SS]_0[S]_0^{\frac{2}{n}-2}[S]^{\frac{n-2}{n}}[S]'.
	\end{equation*}
	Multiplying both sides by the integrating factor $[S]^{\frac{1-n}{n}}$ we obtain
	\begin{equation}
	\frac{\textnormal{d}}{\textnormal{d}t}\bigg( [SI][S]^{\frac{1-n}{n}} \bigg)=\frac{\beta+\gamma}{\beta}[S]^{\frac{1-n}{n}}[S]'-\frac{n-1}{n}[SS]_0[S]_0^{\frac{2}{n}-2}[S]^{-\frac{1}{n}}[S]'.
	\label{eqn:tobeint}
	\end{equation}
	Integrating (\ref{eqn:tobeint}) from $t=0$ to $t=+\infty$, and recalling that, by Proposition \ref{prop:fast}, $[SI]_\infty =0$, we obtain
	\begin{equation}
	-[SI](0)S_0^{\frac{1-n}{n}}=n\frac{\beta+\gamma}{\beta} [S]^{\frac{1}{n}} \bigg|_{t=0}^{+\infty}-[SS]_0[S]_0^{\frac{2}{n}-2}[S]^{\frac{n-1}{n}} \bigg|_{t=0}^{+\infty}.
	\label{eqn:almost}
	\end{equation}
	Since, by assumption, the left-hand side of (\ref{eqn:almost}) is $\mathcal{O}(\epsilon)$, we ignore it, and we consider the right-hand side only. Hence, we find $[S]_\infty$ by solving
	$$
	n\frac{\beta+\gamma}{\beta} [S]^{\frac{1}{n}} \bigg|_{t=0}^{+\infty}=[SS]_0[S]_0^{\frac{2}{n}-2}[S]^{\frac{n-1}{n}} \bigg|_{t=0}^{+\infty},
	$$
	from which we immediately obtain that $[S]_\infty$ is given as a zero of the function $H(x)$ defined in (\ref{eqn:Hsinf}). We now prove that such a zero is unique.\\
	Recall $\frac{[SS]_0}{[S]_0}\leq n$; we have
	$$
	H(0)=[S]_0^{\frac{1}{n}}\left(-n\frac{\beta+\gamma}{\beta}+\frac{[SS]_0}{[S]_0}\right)<0, \quad H([S]_0)=0.
	$$
	Moreover,
	$$
	H'(x)=\frac{\gamma+\beta}{\beta}x^{\frac{1}{n}-1}-\frac{n-1}{n}[SS]_0[S]_0^{\frac{2}{n}-2}x^{-\frac{1}{n}}=x^{\frac{1}{n}-1}\bigg( \frac{\gamma+\beta}{\beta}- \frac{n-1}{n}[SS]_0[S]_0^{\frac{2}{n}-2}x^{1-\frac{2}{n}} \bigg).
	$$
	Recall (\ref{losshyp}). We see that $H'(x)>0$ for
	$$
	x<\bigg( \frac{L[S]_0}{[SS]_0} \bigg)^{\frac{n}{n-2}}[S]_0=:[S]_*([S]_0,[SS]_0).
	$$
	Clearly, $[S]_*=[S]_*([S]_0,[SS]_0)<[S]_0$, since we assumed $[SS]_0>L[S]_0$. Lastly,
	$$
	H'([S]_0)=[S]_0^{\frac{1}{n}-1}\bigg(\frac{\gamma+\beta}{\beta} - \frac{n-1}{n} \frac{[SS]_0}{[S]_0}\bigg)<0 \quad\hbox{ if }\quad [SS]_0>L[S]_0.
	$$
	Hence, $H(x)$ increases on the interval $[0,[S]_*)$, has a positive maximum in $x=[S]_*$, and then decreases towards $0$; in particular, it has a unique zero on the interval $[0,[S]_*)$, and hence in the interval $[0,[S]_0)$.
\end{proof}
\begin{remark}
	Recall (\ref{losshyp}) and Proposition \ref{prop:fast}. Given a couple $([S]_0,[SS]_0)$ in the repelling region $\mathcal{C}_0^R$ above the line $[SS]=L[S]$ (i.e., where $\lambda_5>0$), its image under the fast flow (\ref{eqn:layer}), approximated up to $\mathcal{O}(\epsilon)$ by formulas (\ref{eqn:Hsinf}) and (\ref{eqn:SSinfty}), is in the attracting region $\mathcal{C}_0^A$ below the line $[SS]=L[S]$ (i.e., where $\lambda_5<0$); refer to Figure \ref{fig:loss} for a visualization.
\end{remark} 
\begin{figure}[h!]\centering
	\begin{tikzpicture}
	\node at (0,0){
		\includegraphics[width=0.35\textwidth]{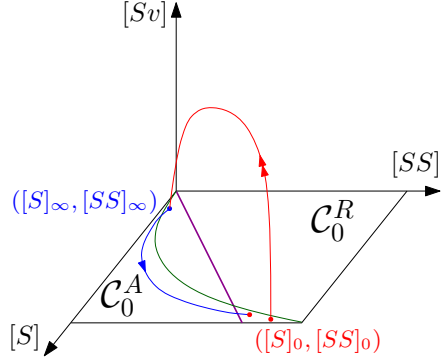}};
	\node at (-2.9,-2.1) {$[S]$};
	\node at (-1.3,2.2) {$[Sv]$};
	\node at (2.3,0.2) {$[SS]$};
	\node at (-1.6,-1.55) {\Large $\mathcal{C}_0^A$};
	\node at (1.2,-0.55) {\Large $\mathcal{C}_0^R$};
	\node at (1,-2.15) {\small \color{red}$([S]_0,[SS]_0)$};
	\node at (-2.1,-0.3) {\small\color{blue}$([S]_\infty,[SS]_\infty)$};
	\end{tikzpicture}
	\caption{Red curve: evolution of the point $([S]_0,[SS]_0)$ under the fast flow. Blue curve: evolution of the point $([S]_\infty,[SS]_\infty)$ under the slow flow. Green curve: curve $[SS]=\alpha([S])$ defined in (\ref{alpha}). Purple line: line of loss of hyperbolicity $[SS]=L[S]$ of the critical manifold of system (\ref{eqn:sist2}), which divides the attracting region $\mathcal{C}_0^A$ and the repelling one $\mathcal{C}_0^R$.}
	\label{fig:loss}
\end{figure}
\begin{remark} Recall (\ref{eqn:basicrep}), and that we assume $R_1>1$. Then  $L=\frac{n(\beta+\gamma)}{\beta(n-1)}=\frac{n}{R_1}<n$. Hence, the purple line $[SS]=L[S]$ in Figure \ref{fig:loss} is always below the line $[SS]=n[S]$. 
\end{remark}

\subsection{Equilibria of the perturbed system}

The following Lemma discusses the equilibria of system (\ref{eqn:sist2}).
\begin{prop}
For $\epsilon > 0$ sufficiently small and $R_0 > 1$, system (\ref{eqn:sist2}) has $2$ equilibria in the relevant region of $\mathbb{R}^5$.\\
\textbf{Disease free equilibrium:}
\begin{equation*}
[S]=1, \quad [I]=0, \quad [SS]=n, \quad [SI]=0, \quad [II]=0.
\end{equation*}
\textbf{Endemic equilibrium:}  to their first order on $\epsilon$ the components are given by:
\begin{equation}
\begin{gathered}\label{eqn:equiend}
[S]=\frac{(n-1)(\gamma +\beta)}{\left(n^2-n-1\right) \beta - \gamma } +\mathcal{O}(\epsilon), \\ [I]=\epsilon \frac{n ((n-2) \beta -\gamma)}{\gamma  \left( \left(n^2-n-1\right) \beta-\gamma \right)}+\mathcal{O}(\epsilon^2), \\
[SS]=\frac{n (\gamma +\beta )^2}{\beta ( \left(n^2-n-1\right) \beta - \gamma )}+\mathcal{O}(\epsilon),
\\ [SI]=\epsilon\frac{n ((n-2)\beta -\gamma )}{\beta  \left( \left(n^2-n-1\right) \beta-\gamma \right)}+\mathcal{O}(\epsilon^2), \\
[II]=\epsilon \frac{n ((n-2) \beta -\gamma)}{\gamma  \left( \left(n^2-n-1\right) \beta-\gamma \right)}+\mathcal{O}(\epsilon^2).
\end{gathered}
\end{equation}
\end{prop}
\begin{proof}
The disease free equilibrium is trivial. The endemic equilibrium is computed by expanding the variables in power series of $\epsilon$, e.g. $[S]=S_0+\epsilon S_1 + \mathcal{O}(\epsilon^2)$, substituting them in system (\ref{eqn:sist2}), equating the right-hand sides to 0 and matching powers of $\epsilon$.
\end{proof}
\begin{remark} Since we assume $R_0=\frac{\beta(n-2)}{\gamma}>1$, recall Remark \ref{othererr} and (\ref{eqn:truebasicrep}), the numerators of $[I]$, $[SI]$ and $[II]$ of (\ref{eqn:equiend}), as well as all the denominators, are strictly positive for $\epsilon>0$ small enough.\end{remark}
We notice that the disease free equilibrium belongs to $\mathcal{C}_0$ defined in (\ref{eqn:critical}), and by computing the corresponding $\lambda_5 = \beta(n-2)-\gamma =\gamma (R_0-1)-\mathcal{O}(\epsilon)>0$, we show that it is unstable.\\
Moreover, we notice that the endemic equilibrium is $\mathcal{O}(\epsilon)$ close to the line $[SS]=L[S]$ defined in (\ref{losshyp}); hence, it approaches it as $\epsilon\rightarrow 0$.

\subsection{Slow manifold}\label{subsec:slow}

Next, we provide a multiple time scale description of the disease-free, or near disease-free states:

\begin{prop}\label{prop:exponn}
The slow manifold of system (\ref{eqn:sist2}) is exponentially close in $\epsilon$ to the critical manifold $\mathcal{C}_0$ given by (\ref{eqn:critical}).
\end{prop}
\begin{proof}
	The invariant manifold $\mathcal{C}_0$ is an invariant manifold also for system (\ref{eqn:sist2}) with $\epsilon>0$: by direct substitution, we have that $[I]'$, $[SI]'$ and $[II]'$ are zero on $\mathcal{C}_0$. Hence, $\mathcal{C}_0$ is invariant and satisfies all the conclusions of Fenichel's theorem, and so it is one possible slow manifold. By Fenichel's theorem, all slow manifolds are exponentially close to each other in the normally hyperbolic region; invariance allows us to extend at least one slow manifold across the line where we do not have normal hyperbolicity, namely $\mathcal{C}_0$.
\end{proof}
\noindent We provide an explicit computation of the slow manifold, expanding it in orders of $\epsilon$, in Appendix A.\\
\\
The slow dynamics on the slow manifold $[I]=[SI]=[II]=0$ are given by:
\begin{align*}
\begin{split}
[S]'={}& \epsilon   (1-[S]),\\
[SS]'={}&2\epsilon  (n[S]-[SS]),
\end{split}
\end{align*}
which, rescaling the system to the slow time variable $\tau=\epsilon t$, becomes 
\begin{align}
\begin{split}
\dot{[S]}={}&   (1-[S]),\\
\dot{[SS]}={}&2  (n[S]-[SS]).
\end{split}
\label{eqn:slow2}
\end{align}
Recall that $[S]_\infty$ and $[SS]_\infty$ are the initial conditions for the slow flow. Solving (\ref{eqn:slow2}) explicitly yields
\begin{align}
\begin{split}
[S](\tau)={}& ([S]_\infty-1)e^{- \tau}+1,\\
[SS](\tau)={}&2 ([S]_\infty-1)n e^{-2\tau} (e^{ \tau}-1)+([SS]_\infty-n)e^{-2 \tau}+n,
\end{split}
\label{eqn:slowsol}
\end{align}
meaning that $[S]\rightarrow 1$, $[SS]\rightarrow n$ exponentially fast, as we would expect, since in the slow dynamics, on the node level, the variable $[R]$ can only decrease, and $[S]$ can only increase.\\
For its importance in the dynamics, we introduce the following notation
\begin{equation}
\Gamma := \{ ([S],[SS])\in [0,1]\times [0,n]| [SS]=n[S]^2  \}.
\label{eqn:slowpar}
\end{equation}

\begin{lemm}\label{unifattr}
	The parabola $\Gamma$ (\ref{eqn:slowpar}) is uniformly attracting for system (\ref{eqn:slow2}).
\end{lemm}
\begin{proof}
	Recall (\ref{eqn:slowsol}). Then, introducing the function $d(\cdot)$ to indicate the distance between a point and the parabola, we have
	\begin{equation*}
	\begin{split}
	d(\tau)=|[SS](\tau)-n[S]^2(\tau)|=&|2n ([S]_\infty-1) e^{- \tau} -2n ([S]_\infty-1) e^{-2\tau}+([SS]_\infty-n)e^{-2  \tau}+n-\\
	-n([S]_\infty -1)^2& e^{-2\tau}-n-2n([S]_\infty -1) e^{-\tau}|=|e^{-2\tau}([SS]_\infty -n[S]_\infty^2)|=e^{-2\tau}d(0),
	\end{split}
	\end{equation*}
	which means that an orbit starting in any point $([S]_\infty, [SS]_\infty)\in (0,1)\times(0,n)$ approaches exponentially fast the parabola $\Gamma$ (\ref{eqn:slowpar}).
\end{proof}
\begin{lemm}
	Consider an orbit starting (i.e. exiting the slow manifold) $\mathcal{O}(\delta_2)$, where $0<\delta_2 \ll 1$, away from the parabola $[SS]=n[S]^2$, in a point with $[S](0)=[S]_0$ in the repelling region of $\mathcal{C}_0$, i.e. satisfying $R_1[S]_0>1$, $[SS]_0>L[S]_0$. Its entry point in the slow flow $[S]_\infty$ after a fast piece is given as the unique zero smaller than $[S]_0$ of
	\begin{equation}
	G(x) =\frac{\beta+\gamma}{\beta}(x^{\frac{1}{n}}-[S]_0^{\frac{1}{n}})- [S]_0^{\frac{2}{n}}x^{1-\frac{1}{n}}+[S]_0^{1+\frac{1}{n}}.
	\label{eqn:newsinf}
	\end{equation}
\end{lemm}
\begin{proof} Notice that, considering Lemma \ref{unifattr}, the assumption of starting close to the parabola is not restrictive.
	The derivation of $G(x)$ is analogous to the derivation of $H(x)$ of Proposition \ref{prop:entryy}, using 
	$$	
	[SS](t)=n[S]_0^{\frac{2}{n}}([S](t))^{\frac{2n-2}{n}}
	$$
	instead of (\ref{eqn:SSofSt}), since we assume $[SS]_0=n[S]_0^2$. The uniqueness of the zero is obtained applying Proposition 3 to this specific initial condition.
\end{proof}
\begin{remark}\label{rem:nbig} Recall (\ref{eqn:SSinfty}). Since we showed that the parabola $\Gamma$ (\ref{eqn:slowpar}) is attracting in the slow flow, we can assume that, after the first slow piece of any orbit, $[SS]_0 = n[S]_0^2 +\mathcal{O}(\delta_1)$, where $0<\delta_1 \ll 1$. We can then rewrite (\ref{eqn:SSinfty}) as
\begin{equation*}
[SS]_\infty = [SS]_0 \bigg( \frac{[S]_\infty}{[S]_0}\bigg)^{\frac{2n-2}{n}}\approx n [S]_0^2 \bigg( \frac{[S]_\infty}{[S]_0}\bigg)^{\frac{2n-2}{n}}=n [S]_\infty^2 \bigg(\frac{[S]_\infty}{[S]_0}\bigg)^{-\frac{2}{n}},
\end{equation*}
where the $\approx$ symbol indicates an $\mathcal{O}(\delta_1)$ error.
For $n$ large enough, the last factor is close to 1, and the entry point for the slow flow is approximately on the parabola.\end{remark}

\subsection{Rescaling}

From now on, we are going to assume $n =\mathcal{O}(1)$.
As we showed in Section \ref{subsec:fast}, under the fast flow eventually $[I]$, $[SI]$ and $[II]$ will be $\mathcal{O}(\epsilon)$; recall (\ref{eqn:const2}), from which we see that $[I]=\mathcal{O}(\epsilon)$ implies $[SI],[II],[IR]=\mathcal{O}(\epsilon)$. Proceeding as in \cite{jardnkojakhmetov2020geometric}, we rescale $[I]=\epsilon[v]$. This implies, using (\ref{eqn:const2}),
\begin{equation*}
[SI]= \epsilon[Sv], \quad [II]=\epsilon[vv].
\end{equation*}
This brings the model, after rearranging the variables, to a singularly perturbed system of ODEs, namely 
\begin{align}
\begin{split}
[S]'={}&-\epsilon \beta [Sv]+ \epsilon   (1-[S]-\epsilon [v]),\\
[SS]'={}&2\epsilon  (n[S]-[SS]-\epsilon[Sv])-2\epsilon \beta\frac{n-1}{n}\frac{[SS][Sv]}{[S]},\\
\epsilon [v]'={}&\epsilon \beta [Sv]- \epsilon \gamma [v],\\
\epsilon [Sv]'={}&-\epsilon (\gamma+\beta)[Sv]+\epsilon^2   (n[v]-[Sv]-[vv])+\epsilon \beta\frac{n-1}{n}[Sv]\bigg(\frac{[SS]}{[S]}-\epsilon \frac{[Sv]}{[S]}\bigg),\\
\epsilon [vv]'={}&2\epsilon \beta[Sv]-2\epsilon \gamma [vv]+\epsilon^2 \beta\frac{n-1}{n} \frac{[Sv]^2}{[S]},
\end{split}
\label{eqn:rescaleEE}
\end{align}
which can be rewritten in a standard form, and rescaled to the slow time scale, denoting now the time derivative with an overdot, giving
\begin{align}
\begin{split}
\dot{[S]}={}&- \beta [Sv]+    (1-[S]-\epsilon [v]),\\
\dot{[SS]}={}&2  (n[S]-[SS]-\epsilon[Sv])-2 \beta\frac{n-1}{n}\frac{[SS][Sv]}{[S]},\\
\epsilon \dot{[v]}={}& \beta [Sv]- \gamma [v],\\
\epsilon \dot{[Sv]}={}&-(\gamma+\beta)[Sv]+\epsilon   (n[v]-[Sv]-[vv])+ \beta\frac{n-1}{n}[Sv]\bigg(\frac{[SS]}{[S]}-\epsilon \frac{[Sv]}{[S]}\bigg),\\
\epsilon \dot{[vv]}={}&2\beta[Sv]-2 \gamma [vv]+\epsilon \beta\frac{n-1}{n} \frac{[Sv]^2}{[S]}.
\end{split}
\label{eqn:rescaleSLOW}
\end{align}
Taking now the $\lim_{\epsilon \rightarrow 0}$, we obtain the system of algebraic-differential equations
\begin{align}
\begin{split}
\dot{[S]}={}&- \beta [Sv]+    (1-[S]),\\
\dot{[SS]}={}&2  (n[S]-[SS])-2 \beta\frac{n-1}{n}\frac{[SS][Sv]}{[S]},\\
0={}& \beta [Sv]- \gamma [v],\\
0={}&-(\gamma+\beta)[Sv]+ \beta\frac{n-1}{n}\frac{[Sv][SS]}{[S]},\\
0={}&2\beta[Sv]-2 \gamma [vv].
\end{split}
\label{eqn:slowADE}
\end{align}
The last three equations of (\ref{eqn:slowADE}) are satisfied for $[v]=[Sv]=[vv]=0$. This is exactly the critical manifold of (\ref{eqn:sist2}), on which the dynamics is described by (\ref{eqn:slow2}).\\
Using (\ref{eqn:slow2}), we can show how $\lambda_5$ changes in time, in the slow flow, by deriving its formulation (\ref{eqn:eigwsign}) with respect to time, obtaining
\begin{equation*}
\dot{\lambda}_5=\beta   \frac{n-1}{n} \frac{2n[S]^2-[SS]([S]+1)}{[S]^2}.
\end{equation*}
This implies that $\lambda_5$ is increasing if $[SS]<\alpha([S])$, where the function $\alpha$ is defined by
\begin{equation}
\alpha(x)=\frac{2nx^2}{x+1}.
\label{alpha}
\end{equation}
\begin{figure}[H]\centering
\begin{tikzpicture}
	\node at (0,0){
		\includegraphics[width=0.35\textwidth]{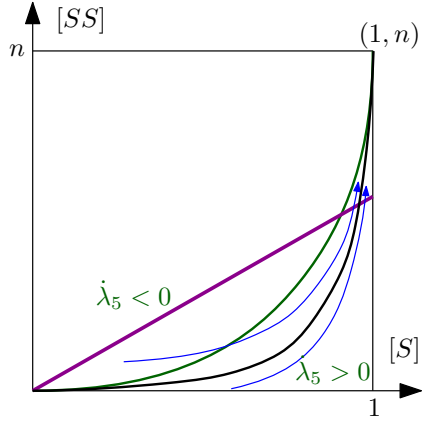}};
	\node at (2.4,-2) {$[S]$};
	\node at (2,-2.75) {$1$};
	\node at (-1.9,2.4) {$[SS]$};
	\node at (-2.75,2) {$n$};
	\node at (2.2,2.2) {$(1,n)$};
	\node at (-1.2,-1.25){\color{NewGGreen}$\dot{\lambda}_5<0$};
	\node at (1.45,-2.25){\color{NewGGreen}$\dot{\lambda}_5>0$};
\end{tikzpicture}
\caption{Sign of the derivative in time of $\lambda_5$ under the slow flow. Blue: sketch of orbit starting above/below the green curve $[SS]=\alpha([S])$. Purple: loss of hyperbolicity line $[SS]=L[S]$ (\ref{losshyp}). Black: the parabola $\Gamma$ (\ref{eqn:slowpar}). Notice that orbits always land below the purple line, which is the region of the rectangle in which the critical manifold is attracting.}
\label{fig:lam5}
\end{figure}%
In Figure \ref{fig:lam5} we visualize the behaviour of two orbits in the slow dynamics. We note that, even if an orbit enters the slow flow in a point below the purple line but above the green curve, i.e. in the region where $\dot{\lambda}_5<0$, it eventually has to cross the green line before crossing the purple curve, since they represent respectively $\dot{\lambda}_5=0$ and $\lambda_5=0$. Hence, any orbit will eventually evolve in the region $\dot{\lambda}_5>0$. We prove the following:
\begin{prop}\label{prop:below}
The subset $\{ ([S],[SS])\in (0,1)\times(0,n)| \dot{\lambda}_5>0 \}$ is forward invariant for system (\ref{eqn:slow2}).
\end{prop}
\begin{proof}
The normal vector to the curve $\alpha([S])$ is given by $\nu=(-\dot{\alpha}([S]),1)$, with
\begin{equation*}
\dot{\alpha}([S])=\frac{2n[S]([S]+2)}{([S]+1)^2}.
\end{equation*}
If we take the scalar product of $\nu$ with the vector field $F$ given by (\ref{eqn:slow2}), we obtain
\begin{equation*}
\nu \cdot F = 2  \bigg( \frac{n(2[S]^3+3[S]^2-[S])}{([S]+1)^2} -[SS]  \bigg)<2  (\alpha([S]) -[SS]),
\end{equation*}
meaning that on the curve $[SS]=\alpha([S])$, this scalar product is negative, hence orbits approaching the curve from below will not cross it. 
\end{proof}

\begin{remark}
	By comparing (\ref{eqn:slowpar}) and (\ref{alpha}), we notice that the curve $\alpha$ is always above the parabola $\Gamma$; hence, by invariance of $\Gamma$ and Proposition \ref{prop:below}, an orbit starting above the parabola will eventually be ``squeezed'' between $\alpha$ and $\Gamma$.
\end{remark}

\subsection{Entry-exit function}\label{subsentrex}

Dividing the last three equations of system (\ref{eqn:rescaleEE}) by $\epsilon$ on both sides, we obtain
\begin{align}
\begin{split}
[S]'={}&\epsilon(- \beta [Sv]+ (1-[S]-\epsilon [v])),\\
[SS]'={}&\epsilon\bigg(2  (n[S]-[SS]-\epsilon[Sv])-2 \beta\frac{n-1}{n}\frac{[SS][Sv]}{[S]}\bigg),\\
[v]'={}& \beta [Sv]- \gamma [v],\\
[Sv]'={}&-(\gamma+\beta)[Sv]+\epsilon   (n[v]-[Sv]-[vv])+ \beta\frac{n-1}{n}[Sv]\bigg(\frac{[SS]}{[S]}-\epsilon \frac{[Sv]}{[S]}\bigg),\\
[vv]'={}&2\beta[Sv]-2 \gamma [vv]+\epsilon \beta\frac{n-1}{n} \frac{[Sv]^2}{[S]}.
\end{split}
\label{eqn:rescaleFAST}
\end{align}
System (\ref{eqn:rescaleFAST}) can be rewritten as
\begin{align}
\begin{split}
x'={}&\epsilon f(x,z)+\epsilon^2 m(z,w),\\
z'={}&zg(x,z)+\epsilon h(x,z,w),\\
w'={}&-Dw+Az+\epsilon l(x,z),
\end{split}
\label{eqn:short}
\end{align}
where we denote $x:= {{[S]}\choose{[SS]}}$, $z:=[Sv]$, and $w:={{[v]}\choose{[vv]}}$. The critical manifold $\mathcal{C}_0=\{z=0, w={{0}\choose{0}}\}$ is invariant for system (\ref{eqn:short}) both when $\epsilon>0$ and $\epsilon=0$. Recall (\ref{eqn:eigwsign}); it is clear that $g(x,0)=\lambda_5\lessgtr 0$ when $x\in \mathcal{C}_0^A$ or $x\in \mathcal{C}_0^R$, respectively.\\
To control the relation between the starting point of the slow dynamics and the transition point back to the fast dynamics, we are going to employ the entry-exit function~\cite{de2016entry}. This tool relies on calculating a fast variational equation along a slow orbit to calculate the exit point from the slow dynamics after a change from fast attraction to fast repulsion has taken place. We now describe this idea in more detail in our current setting, and we apply it to system (\ref{eqn:short}).\\
In the spirit of what was done in \cite{jardnkojakhmetov2020geometric}, we want to apply formula (12)-(13) of \cite{hsu2019relaxation}, in order to obtain more information on the slow part of the dynamics. Consider system (\ref{eqn:short}); the couple $(x,z)$ is in a formulation which allows us to apply the entry exit formula to it, ignoring the variable $w$, since its behaviour does not depend, in the limit as $\epsilon\rightarrow 0$, on the position of $x$ on the critical manifold, and $x$ and $z$ depend on $w$ only at an higher order of $\epsilon$ (second and first, respectively).\\
Recall (\ref{eqn:eigwsign}); from (\ref{eqn:slowsol}), we know that orbits starting in the region in which $\lambda_5<0$ will eventually reach the region in which $\lambda_5>0$, and remain in the latter. 
In the first part of this evolution, the system builds up attraction towards the slow manifold, but after the orbit crosses the loss of hyperbolicity line, the system starts to build up repulsion which will, eventually, compensate the attraction of the first part.\\
We denote with $x_0:=([S]_\infty,[SS]_\infty)$. Then, if $x(\tau;x_0)$ is the solution of
$$
\begin{cases}
\dot{x}{}&=f(x,0,0),\\
x(0){}&=x_0,
\end{cases}
$$
we can implicitly compute the exit time $T_E$ of an orbit on the slow manifold, applying formula (12) of \cite{hsu2019relaxation} to the couple $(x,z)$ of system (\ref{eqn:short}), through the integral
\begin{equation}
\int_{0}^{T_E} g(x(\tau;x_0),0)\textnormal{d}\tau =0.
\label{eqn:exitNO}
\end{equation}
Recall (\ref{eqn:slowsol}); for ease of notation, we introduce $A:=[S]_\infty -1 <0$ and $B:=[SS]_\infty -n<0$. Then, (\ref{eqn:exitNO}) becomes
\begin{equation}
\begin{split}
&\int_{0}^{T_E}\lambda_5(\tau)\textnormal{d}\tau  =\int_{0}^{T_E} \bigg( -(\gamma+\beta)+ \beta \frac{n-1}{n} \frac{[SS](\tau)}{[S](\tau)}  \bigg)\textnormal{d}\tau =\\
&\int_{0}^{T_E}  \bigg( -(\gamma+\beta)+ \beta \frac{n-1}{n}\frac{2A n e^{-2  \tau}(e^{  \tau}-1)+Be^{-2  \tau}+n}{Ae^{-  \tau}+1}   \bigg)\textnormal{d}\tau  =0,
\end{split}
\label{eqn:exitYE}
\end{equation}
which gives the following equation for $T_E$
\begin{equation}
\begin{split}
-(\gamma+\beta)T_E+\beta\frac{n-1}{n}&\bigg(\frac{A^2  n T_E +Ae^{-  T_E}(2An-B)+(B-A(A+2)n)\ln (Ae^{-  T_E}+1)}{A^2 }\\
&-\frac{2An-B}{A }\bigg)=0.
\end{split}
\label{eqn:exitTIME}
\end{equation}
Clearly, $T_E=0$ is a solution of (\ref{eqn:exitTIME}); the integrand of (\ref{eqn:exitYE}), i.e. $\lambda_5$, along the slow flow, is eventually always increasing, recall Proposition \ref{prop:below}; as we remarked, even though it is negative in the first part of the flow, it becomes eventually (and definitely) positive.
\begin{lemm}\label{lemm:finite}
	The exit time $T_E$ is finite for any initial point $([S]_\infty,[SS]_\infty) \in \mathcal{C}_0^A$.
\end{lemm}
\begin{proof}
	Recall (\ref{eqn:exitYE}). For small positive values of $\tau$, $\lambda_5(\tau)<0$, since the slow dynamics begins in the attracting region $\mathcal{C}_0^A$. Hence, for small values of $\tau\geq 0$ the integral
	$$
	\int_{0}^{\tau} \lambda_5(\sigma)\textnormal{d}\sigma <0.
	$$
	From (\ref{eqn:exitTIME}), we observe that 
	$$
	\lim_{T_E \rightarrow +\infty }\int_{0}^{T_E} \lambda_5(\sigma)\textnormal{d}\sigma =+\infty,
	$$	
	hence there exists at least one finite $T_E$ which satisfies (\ref{eqn:exitNO}).
	From our previous analysis, we know that $\lambda_5(\tau)=0$ only once during the slow flow, and it remains positive afterwards; hence, such $T_E$ is unique. 
\end{proof}

\subsection{Application of the entry-exit formula to the parabola}

As we have remarked so far, the parabola (\ref{eqn:slowpar}) is of particular interest for the dynamics, even more so for large values of $n$. Hence, we are interested in understanding the entry-exit relation on this specific invariant set.
We now consider the evolution, under the slow flow, of the point $([S]_\infty,[SS]_\infty)=(0,0)$; with these initial conditions, (\ref{eqn:slowsol}) becomes
\begin{align}
\begin{split}
[S](\tau)={}& 1-e^{-  \tau},\\
[SS](\tau)={}&n+n e^{-2  \tau} -2ne^{- \tau}=n[S]^2(\tau).
\end{split}
\label{eqn:zerozero}
\end{align}
Being able to write $[SS]$ as a function of $[S]$ allows us to compute the exit point for the origin, which in general is not possible, since $\lambda_5$ depends on both slow variables. Combining (\ref{eqn:zerozero}) and (\ref{eqn:exitYE}) we obtain
\begin{equation}
\begin{split}
&\int_{0}^{[S]_1}\bigg( \frac{-(\gamma+\beta)+\beta(n-1)x}{1-x} \bigg) \textnormal{d}x =0\\ &\implies \beta (n-1)(1-[S]_1)+(\gamma-(n-2)\beta)\ln(1-[S]_1)-\beta(n-1)=0,\\
&\implies -\beta (n-1)[S]_1+(\gamma-(n-2)\beta)\ln(1-[S]_1)=0,
\end{split}
\label{eqn:exitzero}
\end{equation}
where $[S]_1$ indicates the exit point of the orbit which starts at the origin.\\
It can easily be shown, by direct substitution, that orbits with initial conditions $([S]_\infty,[SS]_\infty)=([S]_\infty,n[S]_\infty^2)$ evolve, under the slow flow (\ref{eqn:slowsol}), along the curve $[SS]=n[S]^2$; moreover, this follows from Lemma \ref{unifattr}. The exit point of such an orbit can be computed implicitly, with the same procedure as (\ref{eqn:exitzero}).
\begin{lemm}
Orbits entering the slow flow in a point of the form $([S]_\infty,[SS]_\infty)=([S]_\infty,n[S]_\infty^2)$ exit at a point of the form $([S]_1,n[S]_1^2)$, with $[S]_1$ given by
\begin{equation}
-\beta (n-1)[S]_1+(\gamma-(n-2)\beta)\ln(1-[S]_1)=-\beta (n-1)[S]_\infty+(\gamma-(n-2)\beta)\ln(1-[S]_\infty),
\label{eqn:exitpar}
\end{equation}
which can be equivalently rewritten, introducing for ease of notation $C:=((n-2)\beta-\gamma)/(\beta(n-1))$, as
\begin{equation}\label{eqn:aux}
(1-[S]_1)^C e^{[S]_1} = (1-[S]_\infty)^C e^{[S]_\infty}.
\end{equation}
\end{lemm}
\begin{proof}
Straightforward computation from the integral in (\ref{eqn:exitzero}), where we substitute the lower bound of integration $0$ with a generic $[S]_\infty$.
\end{proof}
\begin{lemm}\label{lemm:monot}
If two entry points on the parabola satisfy $[S]_{\infty,1}<[S]_{\infty,2}$, then the corresponding exit points satisfy $[S]_{1,1}>[S]_{1,2}$.
\end{lemm}
\begin{proof} Recall that the parabola is invariant under the slow flow. The entry-exit relation (\ref{eqn:aux}) implicitly defines a function
	$$
	h(x):=(1-x)^Ce^x,
	$$
meaning that the entry-exit relation can be written as $h([S]_\infty)=h([S]_1)$ (see Figure \ref{fig:acca} for a sketch of the function $h$, and a visualization of the argument of this proof). We observe that $h(0)=1$ and $h(1)=0$.
Deriving $h(x)$, we see that
	$$
	h'(x)=(1-x)^{C-1}(1-C-x)e^x>0 \iff x<1-C=\frac{\gamma+\beta}{(n-1)\beta}=\frac{1}{R_1}.
	$$ 
Hence, $h(x)$ is increasing before $x=1/R_1$, decreasing afterwards. This implies that if $[S]_{\infty,1}<[S]_{\infty,2}$ we have that $h([S]_{\infty,1})<h([S]_{\infty,2})$, and the corresponding exit points satisfy $[S]_{1,1}>[S]_{1,2}>1/R_1$.
\end{proof}

\begin{figure}[H]\centering
	\begin{tikzpicture}
	\node at (0,0){
		\includegraphics[width=0.45\textwidth]{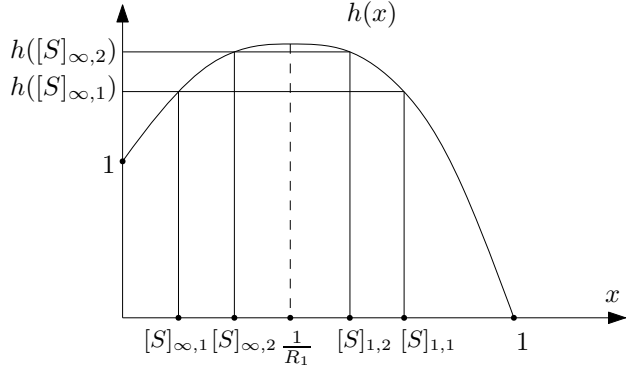}};
	\node at (3.2,-1.75) {$x$};
	\node at (2,-2.35) {$1$};
	\node at (-4.1,1.5) {$h([S]_{\infty,2})$};	
	\node at (-4.1,1) {$h([S]_{\infty,1})$};
	\node at (-3.5,0) {$1$};
	\node at (0.75,-2.35) {\small$[S]_{1,1}$};
	\node at (-0.1,-2.35) {\small$[S]_{1,2}$};	
	\node at (-1,-2.45) {$\frac{1}{R_1}$};
	\node at (-1.7,-2.35) {\small $[S]_{\infty,2}$};
	\node at (-2.6,-2.35) {\small$[S]_{\infty,1}$};	
	\node at (0,2) {$h(x)$};
	\end{tikzpicture}
	\caption{Sketch of the function $h(x)$ used in the proof of Lemma \ref{lemm:monot}.}
	\label{fig:acca}
\end{figure}%
The study of the asymptotic behaviour of system (\ref{eqn:sist2}) is then reduced to two 2-dimensional maps, from $\mathcal{C}_0$ to itself; specifically, we have that $\Pi_1([S]_0,[SS]_0)=([S]_\infty,[SS]_\infty)$, while $\Pi_2([S]_\infty,[SS]_\infty)=([S]_1,[SS]_1)$. We now explain the reasoning under the approximation that $\Pi_1$ and $\Pi_2$ map the parabola $\Gamma$ to itself, and can hence be seen as near one-dimensional (see Figure \ref{fig:conv}); in fact, the occurrence of near one-dimensional return maps is an important theme in multiple time scale systems~\cite{Boldetal,GuckenheimerWechselbergerYoung,KuehnRetMaps,Medvedev}.\\
Next, consider a point with $[S]$ coordinate $[S]_0$, $\mathcal{O}(\epsilon)$ away from the parabola $\Gamma$ (\ref{eqn:slowpar}), in the repelling part of the critical manifold. Its image $[S]_\infty$ under the fast flow, which defines the map $\Pi_1$ sketched in Figure \ref{fig:conv}, is given by (\ref{eqn:newsinf}). We notice that this value depends on both $\beta$ and $\gamma$, as well as on $n$. For $n$ large enough, the entry point in the slow flow will be close to the parabola, as argued in Remark \ref{rem:nbig}; hence, we will be able to compute its exit point $[S]_1$ using (\ref{eqn:exitpar}), which again depends explicitly on all the parameters of the system in a highly non-trivial way. This is different from the SIRWS model studied in \cite{jardnkojakhmetov2020geometric}, in which there was a clear separation between fast parameters, which dictated the fast dynamics, and had no influence on the slow one, and slow parameters, which characterised the viceversa. The map $\Pi_2$ in Figure \ref{fig:conv} sketches the relation between the entry point $[S]_\infty$ and its corresponding exit point $[S]_1$, i.e. (\ref{eqn:newsinf}).
\begin{figure}[h!]\centering
	\begin{tikzpicture}
	\node at (0,0){
		\includegraphics[width=0.5\textwidth]{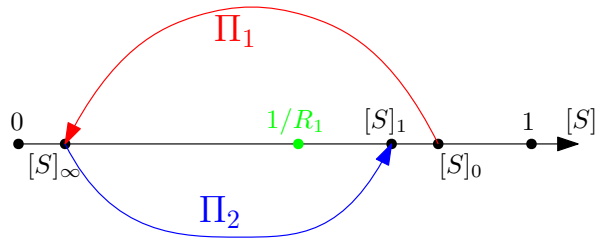}};
	\node at (3.1,0) {$1$};
	\node at (-3.7,0) {$0$};
	\node at (3.8,0) {$[S]$};
	\node at (-1,-1.2) {\Large \color{blue} $\Pi_2$};
	\node at (-0.8,1.2) {\Large \color{red} $\Pi_1$};
	\node at (2.2,-0.6) {$[S]_0$};
	\node at (1.2,0) {$[S]_1$};
	\node at (-3.2,-0.6) {$[S]_\infty$};
	\node at (0,0) {\color{green} $1/R_1$};
	\end{tikzpicture}
	\caption{Sketch of the map which relates $[S]_0$ to $[S]_\infty$ (red) and $[S]_\infty$ to $[S]_1$ (blue). The green dot represents the value $1/R_1$:  the epidemics can only start for values of $[S]_0>1/R_1$.}
	\label{fig:conv}
\end{figure}\\
Depending on the relative position of $[S]_0$ and $[S]_1$, we might be able to deduce the asymptotic behaviour of the system. However, the high dimensionality of the layer equation and the complex implicit relation between $[S]_0$ and $[S]_\infty$ hinders the analysis of the system with non-numerical tools. See Figure \ref{fig:compar} for comparisons between formula (\ref{eqn:newsinf}) and direct integration of the layer system (\ref{eqn:layer}).
\begin{figure}[h!]\centering	
	\begin{subfigure}[t]{.32\textwidth}\centering
		\includegraphics[width=\textwidth]{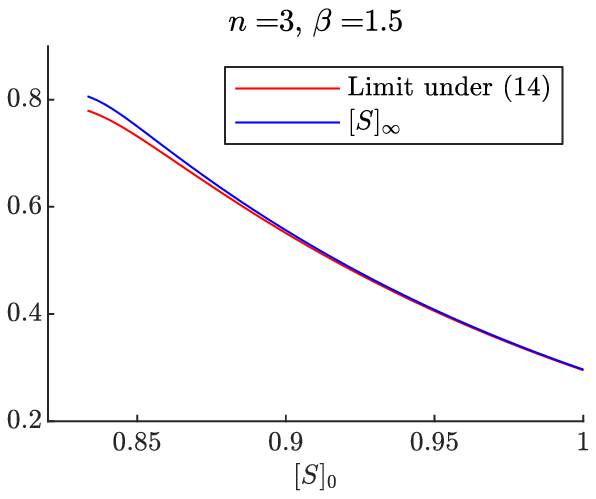}
		\caption{}
		\label{fig:n_3}
	\end{subfigure}\hfill
	\begin{subfigure}[t]{.32\textwidth}\centering
		\includegraphics[width=\textwidth]{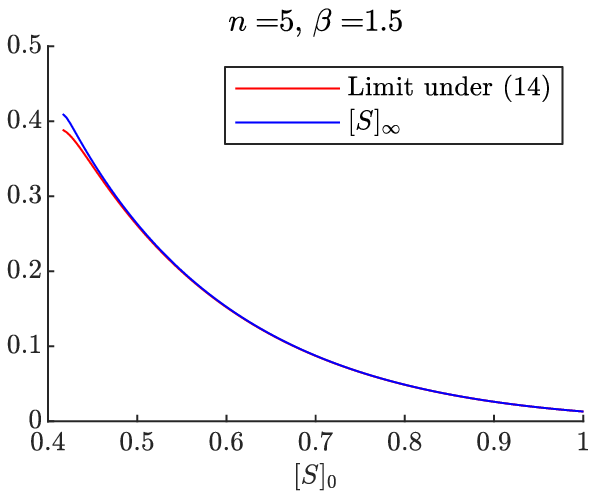}
		\caption{}
		\label{fig:n_5}
	\end{subfigure}\hfill
	\begin{subfigure}[t]{.32\textwidth}\centering
		\includegraphics[width=\textwidth]{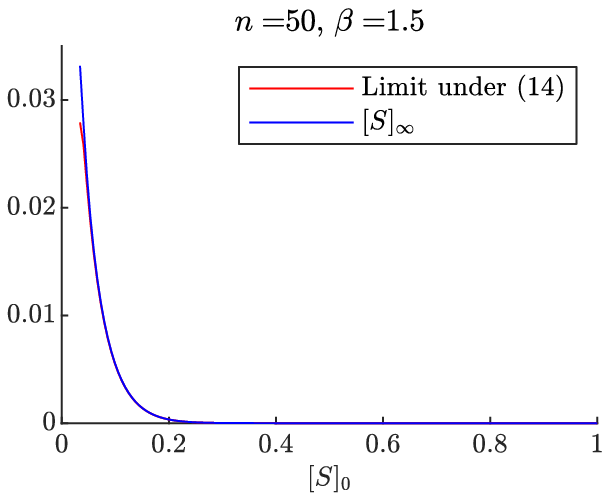}
		\caption{}
		\label{fig:n_50}
	\end{subfigure}
	\caption{Comparison of the limit value of $[S]$ as $t\rightarrow \infty$ of system (\ref{eqn:layer}) (red) and formula (\ref{eqn:newsinf}) (blue). We set $[I]_0=[SI]_0=[II]_0=0.001$, $\gamma =1$. With the values of the parameters of (a) (respectively, (b) and (c)), $1/R_1\approx 0.833$ (resp., $0.417$ and $0.034$), and we only consider values of $[S]_0\geq 1/R_1$, for which the epidemics can start.}
	\label{fig:compar}
\end{figure}\\
We proceed now to a bifurcation analysis of system (\ref{eqn:sist2}), and finally, with a technique similar to the one detailed \cite[Sec.~3.4.1]{jardnkojakhmetov2020geometric}, to numerically investigate the existence of periodic orbits by concatenation of fast and slow pieces. We stress the versatility of the numerical argument we present, which is similar to the one we used in \cite{jardnkojakhmetov2020geometric}, applied now to a higher dimensional system. 

\section{Bifurcation analysis and numerical simulations}\label{sec:bifurc}

In this section, we carry out a bifurcation analysis for the behaviour of system (\ref{eqn:sist2}), which will then be verified by numerical simulations and by a geometrical argument. Bifurcation analysis is done on system (\ref{eqn:sist2}), which for small values of $\epsilon$ is stiff (as we showed in Proposition \ref{prop:exponn}, the slow manifold is exponentially close to the critical manifold), while the numerical simulation concern a combination of systems (\ref{eqn:layer}) and (\ref{eqn:slow2}), which are both non-stiff.\\
It is important to notice that, even though the layer system (\ref{eqn:layer}) converges to the critical manifold forwards in time, the slow flow (\ref{eqn:slow2}) would converge to the point $([S],[SS])=(1,n)$ if we let it evolve freely; the derivation of the exit time (\ref{eqn:exitTIME}) is fundamental, in this setting, to carry out a meaningful numerical exploration of the model.\\
Without loss of generality, we set $\gamma$, which is the inverse of the average infection interval, to 1; this simply amounts to an $\mathcal{O}(1)$ rescaling of time, and we rescale the other parameters accordingly, keeping however the same symbols, for ease of notation. System (\ref{eqn:sist2}) then has only three parameters, namely $\epsilon$, $n$ and $\beta$.\\
Using MatCont \cite{MatCont}, we are able to completely characterize system (\ref{eqn:sist2}) through numerical bifurcation analysis. We only consider the first octant of $\mathbb{R}^3$, for the biological interpretation of the parameters. Numerical analysis shows the existence of a Hopf surface $\Sigma$, whose ``skeleton'' is depicted in Figure \ref{fig:surfa}. For values of the parameters between the plane $\epsilon=0$ and $\Sigma$, the system exhibits a stable limit cycle, while for values above $\Sigma$, the system exhibits convergence to the endemic equilibrium (\ref{eqn:equiend}). Our bifurcation analysis suggests the existence of a value $\epsilon^* \approx 0.18$ such that, for $\epsilon>\epsilon^*$, the system only exhibits convergence to the endemic equilibrium, regardless of the values of $\beta$ and $n$.
\begin{figure}[h!]\centering
	\begin{tikzpicture}
	\node at (0,0){
		\includegraphics[width=0.75\textwidth]{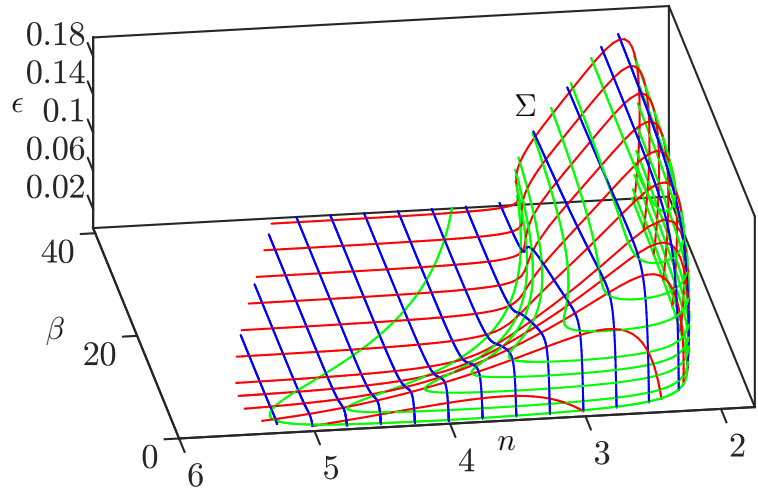}};
	\node at (1.5,-3) {\large $n$};
	\node at (-5,1.5) {\large $\epsilon$};
	\node at (-4.5,-1.5) {\large $\beta$};
	\node at (1.75,1.5) {\large$\Sigma$};
	\end{tikzpicture}
	\caption{A skeleton of the bifurcation surface $\Sigma$. Green (respectively, red and blue) curves correspond to constant values of $\epsilon$ (respectively, $\beta$ and $n$). We notice that, for values of $n\geq 6$, system (\ref{eqn:sist2}) converges to the endemic equilibrium (\ref{eqn:equiend}) regardless of the value of $\epsilon$ and $\beta$.}
	\label{fig:surfa}
\end{figure}\\
To make Figure \ref{fig:surfa} more readable, we provide intersections of the surface $\Sigma$ with some planes $n=k$ (Figure \ref{fig:enn}), $\beta=k$ (Figure \ref{fig:tau}), and finally $\epsilon = k$ (Figure \ref{fig:eps}).
\begin{figure}[h!]
	\begin{subfigure}[t]{.475\textwidth}\centering
		\begin{tikzpicture}
		\node at (0,0){\includegraphics[width=\textwidth]{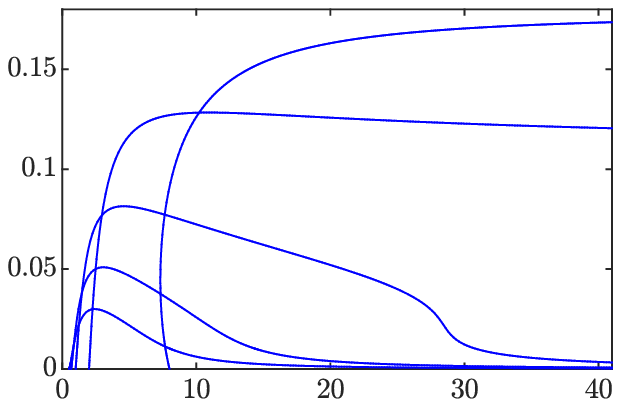}};
		\node at (-3,-0.25){$\epsilon$};
		\node at (-0.5,-2.2){$\beta$};
		\node at (-2,-1.6){\color{blue} \small $4$};
		\node at (-0.9,-1.2){\color{blue} \small $3.5$};
		\node at (0,-0.5){\color{blue} \small $3$};	
		\node at (2.4,0.4){\color{blue} \small $2.5$};
		\node at (2.5,1.4){\color{blue} \small $2.125$};
		\end{tikzpicture}
		\caption[]{Intersections of the surface $\Sigma$ with planes $n=k$. The values of $n$ are indicated near the corresponding curves.}
		\label{fig:enn}
	\end{subfigure}\hfill
	\begin{subfigure}[t]{.475\textwidth}\centering
		\begin{tikzpicture}
		\node at (0,0){\includegraphics[width=\textwidth]{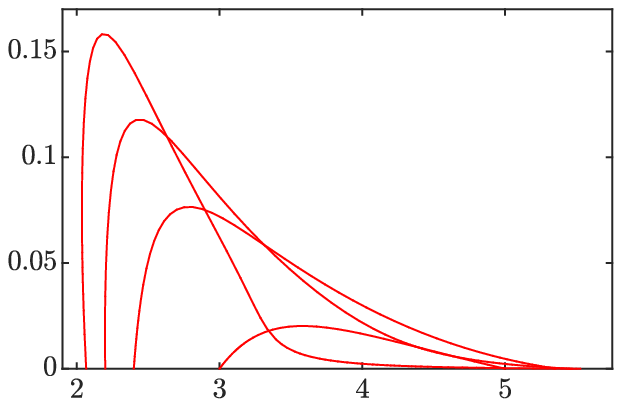}};
		\node at (-3,-0.25){$\epsilon$};
		\node at (-0.5,-2.2){$n$};
		\node at (-0.3,-1.2){\color{red} \small $1$};
		\node at (-1.4,-0.7){\color{red} \small $2.5$};	
		\node at (-1.2,0.4){\color{red} \small $5$};
		\node at (-1.6,1.1){\color{red} \small $15$};
		\end{tikzpicture}	
		\caption[]{Intersections of the surface $\Sigma$ with planes $\beta=k$. The values of $\beta$ are indicated near the corresponding curves.}		
		\label{fig:tau}
	\end{subfigure}
	\caption{A subset of the blue and red curves from Figure \ref{fig:surfa}.}
	\label{fig:bifs}
\end{figure}

\begin{figure}[h!]\centering
	\begin{tikzpicture}
	\node at (0,0){
		\includegraphics[width=0.65\textwidth]{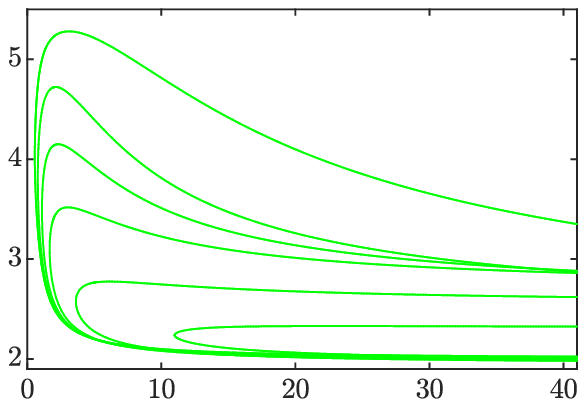}};
	\node at (-4,-0.5){$n$};
	\node at (1,-3){$\beta$};	
	\node at (-1,-1.8){\color{green} \small $0.15$};
	\node at (-2,-1.2){\color{green} \small $0.1$};	
	\node at (-2.8,-0.75){\color{green} \small $0.05$};
	\node at (-2.8,0.05){\color{green} \small $0.025$};	
	\node at (-2,0.8){\color{green} \small $0.01$};
	\node at (-1.5,1.8){\color{green} \small $0.001$};
	\end{tikzpicture}
	\caption{Intersections of the surface $\Sigma$ with planes $\epsilon =k$. The values of $\epsilon$ are indicated near the corresponding curves.}
	\label{fig:eps}
\end{figure}
As in \cite{jardnkojakhmetov2020geometric}, we see an expansion of the parameter region which exhibits stable limit cycles as $\epsilon$ decreases, see Figure \ref{fig:eps}. This means that, as $\epsilon$ decreases, i.e. as the ratio between the average lengths of the infectious phase and the immunity interval decreases, we are more likely to observe occurrence of stable limit cycles in the disease dynamics. We do not observe, however, a divergence in the $n$ direction, as the limit as $\epsilon \rightarrow 0$ of the surface contained in the green curves of Figure \ref{fig:eps} is still bounded.\\
Counter-intuitively from a
biological viewpoint, from which one would expect a greater diffusion of an epidemic in a population consisting of more social individuals, our numerical exploration of system (\ref{eqn:sist2}) shows that limit cycles are only possible for small values of $n$ (specifically $3$, $4$ and $5$). This means that, for a disease with small enough ratio between the infection period and the immunity window (i.e., $\epsilon$), each individual in the population having really few direct neighbours can lead, depending on the force of infection $\tau$, to periodic outbreaks, while having more than $5$ drives the population towards the unique endemic equilibrium. The homogeneous mixing hypothesis can be interpreted, in this network setting, as having $n$ large. In this regard, the bifurcation analysis is in agreement with the results of \cite{jardnkojakhmetov2020geometric}, in which the SIRS model with homogeneous mixing is characterized by convergence towards the endemic equilibrium, as long as $R_0>1$; recall that the endemic equilibrium (\ref{eqn:equiend}) is characterized by a quantity of infected which is $\mathcal{O}(\epsilon)$ small.\\
In order to verify the accuracy of the surface $\Sigma$, we investigate the system via a numerical implementation of the same geometrical argument used in \cite[Sec. 3.4.1]{jardnkojakhmetov2020geometric}. There, we numerically showed the existence of a candidate orbit by concatenating heteroclinic orbits of the layer equation, from the critical manifold to itself, and orbits of the slow flow, truncating each at the corresponding exit time. The system studied in \cite{jardnkojakhmetov2020geometric} was 3-dimensional, but the slow flow evolved on a 2-dimensional plane in $\mathbb{R}^3$; as we showed thus far, system (\ref{eqn:sist2}) is characterized by a 2-dimensional slow manifold, as well. We now briefly recall the construction of the geometrical argument.\\
We fix $\epsilon = 0$, $n=4$, and vary $\beta$ to be below and above $\Sigma$, respectively; we compare the results in Figure \ref{fig:numerics}. In both cases, a candidate starting point for a periodic orbit was found by iterating multiple times the layer system (\ref{eqn:layer}) and the slow flow (\ref{eqn:slow2}), stopped when the slow piece of the orbit reached its exit time (\ref{eqn:exitTIME}).
\begin{figure}[htb!]
	\centering
	\begin{minipage}{\textwidth}
		\centering
		\begin{subfigure}{.45\textwidth}
			\centering
			\includegraphics[width=0.95\textwidth]{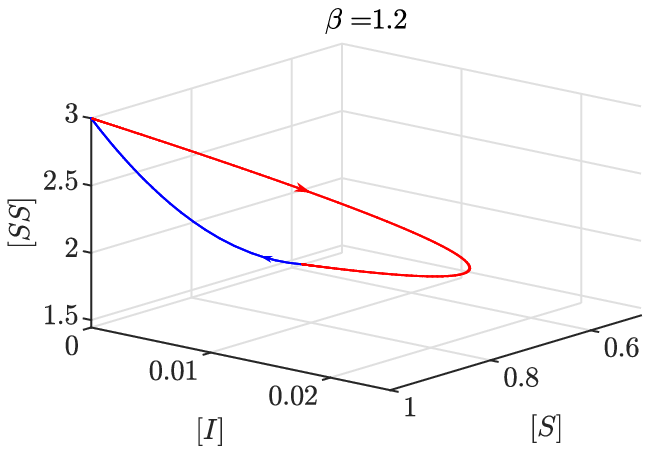}
			\caption{}
		\end{subfigure}
		\begin{subfigure}{0.45\textwidth}
			\centering
			\includegraphics[width=0.95\textwidth]{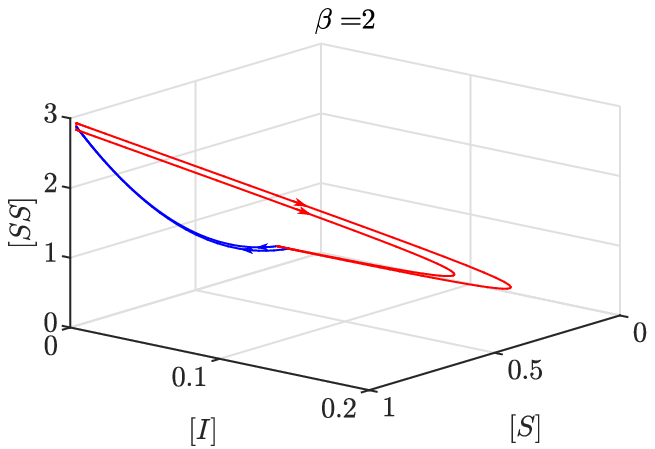}
			\caption{}
		\end{subfigure}%
	\end{minipage}\\
	\begin{minipage}{\textwidth}
		\centering
		\begin{subfigure}{.45\textwidth}
			\centering
			\includegraphics[width=0.95\textwidth]{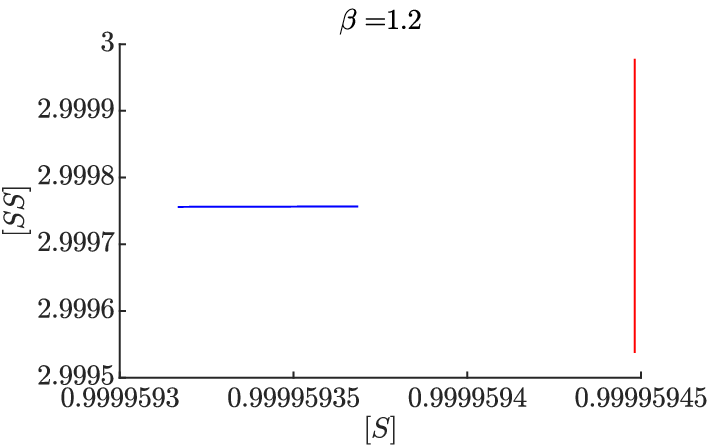}
			\caption{}
		\end{subfigure}
		\begin{subfigure}{0.45\textwidth}
			\centering
			\includegraphics[width=0.95\textwidth]{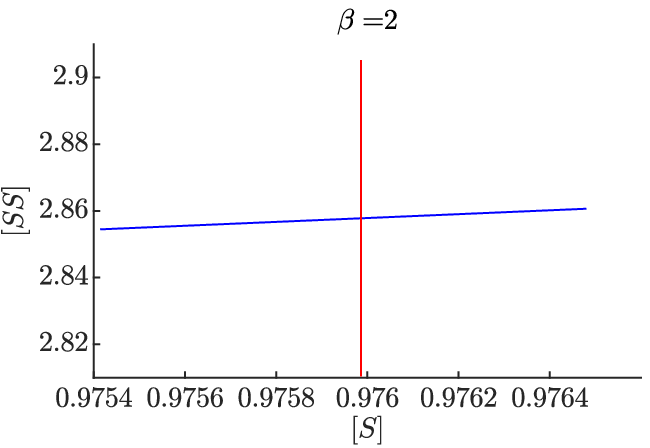}
			\caption{}
		\end{subfigure}%
	\end{minipage}\\
	\begin{minipage}{\textwidth}
		\centering
		\begin{subfigure}{.45\textwidth}
			\centering
			\includegraphics[width=0.95\textwidth]{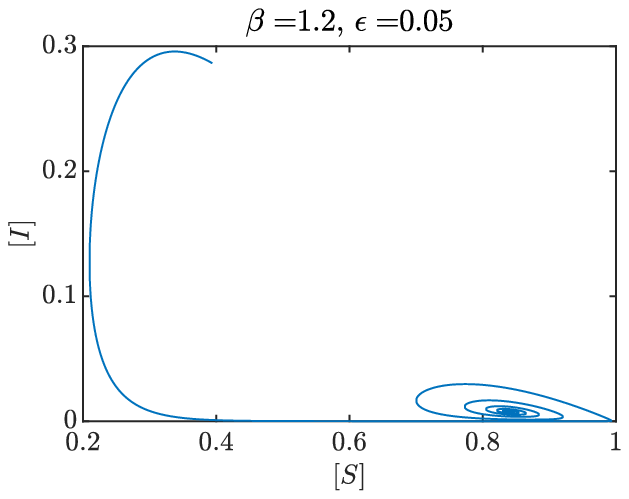}
			\caption{}
		\end{subfigure}
		\begin{subfigure}{0.45\textwidth}
			\centering
			\includegraphics[width=0.95\textwidth]{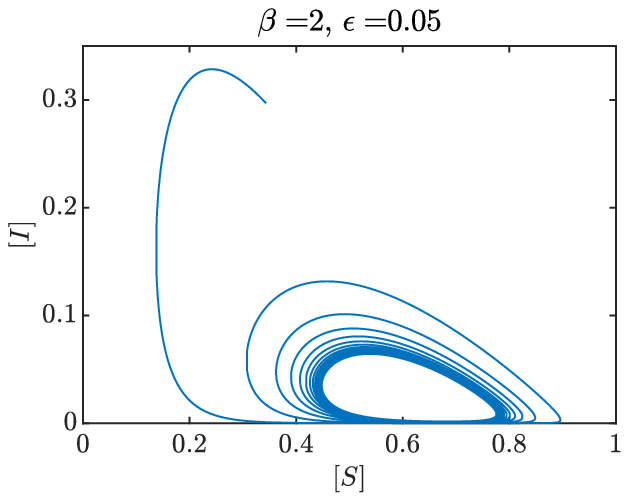}
			\caption{}
		\end{subfigure}%
	\end{minipage}
	\caption{Numerical illustration of the effect of changing $\beta$ on the system dynamics. Figures (a) and (b): evolution under the layer system (red) of a small interval $J_1$, in the $[SS]$ direction. Its image defines the entry interval $J_2$ on the critical manifold; evolution of each point of $J_2$ under the slow flow (blue), stopped at its exit time, giving the exit interval $J_3$. Notice that the blue curves lie on the $[S],[SS]$ plane, while the red curves represent a fast excursion in the region $[I],[SI],[II]>0$. Figures (c) and (d): zoom on the relative position of $J_1$ and $J_3$ on the critical manifold. Figures (e) and (f): projections on the $[S]-[I]$ plane of numerical simulations of system (\ref{eqn:sist2}) from a random point, exhibiting convergence to the endemic equilibrium and the stable limit cycle, respectively.}
	\label{fig:numerics}
\end{figure}
Once we have obtained the candidate value for $[S]_0$, we take a small interval $J_1$ in the $[SS]$ coordinate around its corresponding value $[SS]_0$, and we map it forward in time to obtain an interval of starting points for the slow flow, $J_2$. Finally, we map $J_2$ forward in time, stopping each orbit at the corresponding exit time, obtaining a third interval  $J_3$. In \cite{jardnkojakhmetov2020geometric}, we argued that if $J_3$ intersects $J_1$ transversally, then the perturbed system, for $\epsilon>0$ small enough, exhibits stable limit cycles.\\
Figures \ref{fig:numerics} (a) and (c) depict the numerical realization of the two limit systems for a couple $(n,\beta)$ for which we do not expect limit cycles: indeed, $J_1$ and $J_3$ (respectively, the vertical red line and the blue line in (b)) do not intersect, and bifurcation analysis confirms that, for this choice of the parameters, there are no limit cycles, but global convergence to the endemic equilibrium.\\
Figures \ref{fig:numerics} (b) and (d), instead, depict the numerical realization of the two limit systems for a couple $(n,\beta)$ for which we do expect limit cycles: indeed, $J_1$ and $J_3$ intersect, and bifurcation analysis confirms that, for this choice of the parameters, there is a stable limit cycle. Since the underlying idea is the same as \cite[Sec. 3.4.1]{jardnkojakhmetov2020geometric}, we refer to that for a more in-depth explanation of the method.\\
Figures \ref{fig:numerics} (e) and (f), finally, are projections on the $[S]-[I]$ plane of orbits of system (\ref{eqn:sist2}), starting from a random initial point. As we expected, for $\epsilon$ sufficiently small, the perturbed system exhibits either convergence to equilibrium, as the combination of the two limit systems showed in Figures \ref{fig:numerics} (a) and (c), or towards a stable limit cycle, as argued from (b) and (d).\\
This numerical analysis shows that there is an interval around $\beta \sim 2$ for which periodic orbits of \eqref{eqn:sist2} exists, for $\epsilon>0$ sufficiently small.

\section{Summary and Outlook}\label{sec:conclusions}

We have analysed the behaviour of a model for epidemics on networks, given in a nonstandard singularly perturbed form, after reducing its dimension exploiting multiple conserved quantities.

Even though the model derived from the SIRS model studied in \cite{jardnkojakhmetov2020geometric}, which is characterized by global convergence to equilibrium, our bifurcation analysis and geometric numerical argument show that, for a significant open subset of the parameter space, the network generalization exhibits stable limit cycles. The main characteristic of this set is the value of $n$, the number of neighbours every individual has, which must be between $3$ and $5$ included. In practical terms, this is not a major restriction as most real-world networks have sub-networks, where individuals have around three to five very close friends. It is clear that there is further motivation to intensify the investigation of more complex compartment networks with techniques from GSPT, since dropping the homogeneous mixing hypothesis unveiled asymptotic behaviours which are impossible in the corresponding system studied without network structure. In particular, it would be interesting to check whether the periodic solutions identified in the pair-approximation model of a network can be detected also in simulations of the original network model.

Moreover, the analysis of this network generalization of the SIRS model studied in \cite{jardnkojakhmetov2020geometric} qualitatively confirmed its results, since for large values of $n$ (in the homogeneous mixing hypothesis, $n=N-1$, which is by assumption large), the system only exhibits convergence towards the endemic equilibrium.

We stress the versatility of our geometric procedure, which gives us a numerical intuition of the asymptotic behaviour of a stiff system, i.e. system (\ref{eqn:sist2}) with $0<\epsilon\ll 1$, without having to \emph{actually} integrate it, but through simple integration of the corresponding two non-stiff limit systems, which we derived through the use of GSPT. This is particularly important for the high(er) dimensionality of the system, which hinders analytical results on the perturbed system. In particular, the same strategy is likely to generalize to more complicated network-based ODE models derived from moment closure.

Furthermore, it would be interesting to rigorously investigate how the system changes as we let $n\rightarrow +\infty$. One intermediate step between having two independent perturbation parameters (\emph{i.e.}, $\epsilon \rightarrow 0$ and $n \rightarrow +\infty$) could be to couple $n$ and $\epsilon$, for example taking $n=\mathcal{O}(1/\epsilon^\alpha)$, for some $\alpha>0$. However, this goes beyond the scope of this project, and we leave this as a prompt for future research.\\

\textbf{Acknowledgments:} The work of HJK is partially funded by the Alexander-von-Humboldt Foundation. CK would like to thank the VolkswagenStiftung for support via a Lichtenberg Professorship and the add-on grant ``Corona Crises and Beyond''. MS would like to thank the University of Trento for supporting his research stay at the Technical University Munich, and Eva Loprieno for her aid with Adobe Illustrator, used in Figure \ref{fig:flux}.

\bibliographystyle{plain} \small
\bibliography{biblio}

\section{Appendix A}\label{sec:appendix}
Recall Proposition \ref{prop:exponn}. In this section, we explicitly show that the slow manifold of system (\ref{eqn:sist2}) is exponentially close to the critical manifold $\mathcal{C}_0$ (\ref{eqn:critical}).
\begin{proof}
First of all, we notice that $[I]=\mathcal{O}(\epsilon)$ implies $[SI],[II],[IR]=\mathcal{O}(\epsilon)$; recall (\ref{eqn:const2}).
Proceeding as in \cite{taghvafard2019geometric}, we propose the expansion
\begin{subequations}
\begin{align}
[I]={}&f_1([S],[SS])\epsilon+\mathcal{O}(\epsilon^2),\label{eqn:expans1}\\ [SI]={}&f_2([S],[SS])\epsilon+\mathcal{O}(\epsilon^2),\label{eqn:expans2}\\ [II]={}&f_3([S],[SS])\epsilon+\mathcal{O}(\epsilon^2),\label{eqn:expans3}
\end{align}\label{eqn:expans}%
\end{subequations}
\noindent where the functions $f_i$ are as smooth as necessary. We use these expansions in the respective equations for $[I]', [SI]', [II]'$ in system (\ref{eqn:sist2}), and match the corresponding powers of $\epsilon$.\\
Here we show the details with $[I]$. For ease of notation, we omit arguments of the functions $f_i$ everywhere. We need to solve
\begin{equation*}
[I]' = \bigg(\frac{\partial f_1}{\partial [S]} [S]' + \frac{\partial f_1}{\partial [SS]}[SS]' \bigg)\epsilon +  \mathcal{O}(\epsilon^2) = \beta [SI] - \gamma [I],
\end{equation*}
which becomes
\begin{align}
\begin{split}
&\bigg(\frac{\partial f_1}{\partial [S]}(-\beta f_2 \epsilon+ \epsilon  (1-[S]-f_1 \epsilon) + \mathcal{O}(\epsilon^2)) + \frac{\partial f_1}{\partial [SS]}(2\epsilon (n[S]-[SS]-f_2 \epsilon)-2\beta\frac{n-1}{n}\frac{[SS]f_2 \epsilon}{[S]})\bigg)\epsilon  \\ &+\mathcal{O}(\epsilon^2) = \beta f_2 \epsilon -\gamma f_1 \epsilon + \mathcal{O}(\epsilon^2).
\end{split}
\label{eqn:partial}
\end{align}
The LHS of (\ref{eqn:partial}) is $\mathcal{O}(\epsilon^2)$, while RHS of (\ref{eqn:partial}) is $\mathcal{O}(\epsilon)$; this means that, at first order in $\epsilon$, we have to solve RHS = 0 at first order in $\epsilon$, i.e. ignoring the contribution which is $\mathcal{O}(\epsilon^2)$.\\
From the equation for $[I]'$ we see that
\begin{equation*}
0=\beta f_2 \epsilon - \gamma f_1 \epsilon  \implies f_1 = \frac{\beta}{\gamma} f_2.
\end{equation*}
The same arguments can be applied for $[SI]$ and $[II]$. From the equation for $[SI]'$ we have
\begin{equation*}
0=-(\gamma+\beta) f_2 \epsilon + \epsilon^2 (nf_1-f_2-f_3) +\frac{n-1}{n} \beta f_2 \epsilon \bigg(\frac{[SS]}{[S]} - \frac{f_2 \epsilon}{[S]}\bigg)  \implies  f_2= 0 \implies  f_1= 0.
\end{equation*}
From the equation for $[II]'$
\begin{equation*}
0=2\beta \epsilon f_2-2\gamma \epsilon f_3+2\epsilon^2\beta\frac{n-1}{n}\frac{f_2^2}{[S]} \implies f_3= 0.
\end{equation*}
This shows that, at first order in $\epsilon$, $f_1 = f_2 = f_3 = 0$. So, in the first order in $\epsilon$, the slow manifold is still $[I]=[SI]=[II]=0$.\\
We now prove by induction that, for any $k \in \mathbb{N}$, the slow manifold is exactly 0 in the expansion up to $\epsilon^k$. By assumption, we can write $[I]=g_1([S],[SS],[SR])\epsilon^k+\mathcal{O}(\epsilon^{k+1})$, $[SI]=g_2([S],[SS],[SR])\epsilon^k+\mathcal{O}(\epsilon^{k+1})$, $[II]=g_3([S],[SS],[SR])\epsilon^{k}+\mathcal{O}(\epsilon^{k+1})$.\\
Proceeding as above, all the LHSs will be $\mathcal{O}(\epsilon^{k+1})$, while the RHSs will be $\mathcal{O}(\epsilon^k)$, meaning we still have to solve RHS = 0.\\
From the equation for $[I]'$ (omitting, once again, all the arguments of $g_i$ everywhere, for ease of notation):
\begin{equation*}
0=\beta g_2 \epsilon^k - \gamma g_1 \epsilon^k  \implies g_1 = \frac{\beta}{\gamma} g_2.
\end{equation*}
From the equation for $[SI]'$:
\begin{equation*}
\begin{split}
&0=-(\gamma+\beta) g_2 \epsilon^k + \epsilon^{k+1} (ng_1-g_2-g_3) +\frac{n-1}{n} \beta g_2 \epsilon^k \bigg(\frac{[SS]}{[S]} - \frac{g_2 \epsilon^k}{[S]}\bigg)  \\& \implies g_2= 0 \implies g_1= 0.
\end{split}
\end{equation*}
From the equation for $[II]'$:
\begin{equation*}
0=2\beta g_2 \epsilon^{k} - 2\gamma g_3 \epsilon^{k} +\beta\frac{n-1}{n} \epsilon^{2k}\frac{g_2^2 }{[S]}  \implies g_3 = 0.
\end{equation*}
This shows that the slow manifold is exponentially close in $\epsilon$ to the critical manifold $[I]=[SI]=[II]=0$.
\end{proof}
\end{document}